\documentclass{amsart}
\usepackage{amssymb}
\usepackage{amsmath}
\usepackage{color}
\usepackage{graphicx}
\def\P{{\mathbb P}}
\def\R{{\mathbb R}}

\def\N{{\mathbb N}}

\def\F {{\mathcal F}}

\newtheorem{theo}{Theorem}
\newtheorem{prop}{\indent Proposition}
\newtheorem{lem}{\indent Lemma}

\newtheorem{ass}{Assumption}

\title[Large deviations for interacting diffusions]{Large deviations for cascades of diffusions arising in oscillating systems of interacting Hawkes processes.}
\date{September 8, 2017}
\author{E. L\"ocherbach }

\address{E. L\"ocherbach: CNRS UMR 8088, D\'epartement de Math\'ematiques, Universit\'e de Cergy-Pontoise,
2 avenue Adolphe Chauvin, 95302 Cergy-Pontoise Cedex, France.}

\email{eva.loecherbach@u-cergy.fr}



\begin{document}

\maketitle

\begin{abstract}
We consider oscillatory systems of interacting Hawkes processes introduced in \cite{EvaSusanne} to model multi-class systems of interacting neurons together with the diffusion approximations of their intensity processes. This diffusion, which incorporates the memory terms defining the dynamics of the Hawkes process, is hypo-elliptic. It is given by a high dimensional chain of differential equations driven by $2-$dimensional Brownian motion. We study the large-population-, i.e.,\ small noise-limit of its invariant measure for which we establish a large deviation result in the spirit of Freidlin and Wentzell. 
\end{abstract}

\section{Introduction}
The aim of this paper is to study oscillatory systems of interacting Hawkes processes and their long time behavior. This study has been started in Ditlevsen and L\"ocherbach \cite{EvaSusanne} where multi-class systems of Hawkes processes with mean field interactions  have been introduced as microscopic models for spike trains of interacting neurons. In the large population limit, i.e.\ on a macroscopic scale, such systems present {\it oscillations}. In the present paper we concentrate on the finite population process and its large deviation properties. In particular, we will be interested in its deviations from limit cycles, i.e.\ from the {\it typical oscillatory behavior} of the limit process. 

We consider two populations of particles, the first composed by $ N_1  ,$  the second by $N_2$ particles. The total number of particles in the system is $N = N_1  +N_2$. The activity of each particle is described by a counting process $ Z^{N}_{k, i} (t), {1 \le k \le 2 , 1 \le i \le N_k} , t \geq 0,$ recording the number of ``actions'' of the $i$th particle belonging to population $k$  during the interval $ [0, t ]. $ Such ``actions'' can be ``spikes'' if we think of neurons, it can be ``transactions'', if we think of economical agents. 
The sequence of counting processes $ ( Z^{N}_{k,i} ) $ is characterized by its intensity processes $ (\lambda^N_{k, i} ( t)  ) $ which are informally defined through the relation 
$$ \P ( Z^{N}_{k, i } \mbox{ has a jump in ]t , t + dt ]} | \F_t ) = \lambda^N_{k } (t)  dt , {1 \le k \le 2 , 1 \le i \le N_k} ,$$
where $ \F_t = \sigma (  Z^{N}_{k, i} (s)  , \, s \le t , {1 \le k \le 2 , 1 \le i \le N_k} ) ,$ and where
\begin{equation}\label{eq:intensity}
\lambda^N_{1} ( t)  = f_1  \left( \frac{1}{N_2}\sum_{1 \le j \le N_2}     \int_{]0 , t [} h_{12 }( t-s) d Z^N_{ 2,j} (s)  \right) 
\end{equation}
and 
\begin{equation}\label{eq:intensity2}
\lambda^N_{2} ( t)  = f_2  \left( \frac{1}{N_1}\sum_{1 \le j \le N_1}     \int_{]0 , t [} h_{21 }( t-s) d Z^N_{ 1,j} (s)  \right) .
\end{equation}
The function $f_k$ is called the jump rate function of population $k,$ and the functions $h_{12}, h_{21} $ are the ``memory'' or ``interaction''  kernels of the system. Note that $  \lambda^N_{1} ( t)  $ and $ \lambda^N_{2} ( t) $ encode the interactions of the system and that the way the intensities are defined, particles belonging to the first population depend only on the past jumps of the particles belonging to the second population, and vice versa. In particular, no self-interactions are included in our model.

The form of the intensities \eqref{eq:intensity} is the typical form of the intensity of a multivariate nonlinear Hawkes process. Hawkes processes have been introduced by Hawkes \cite{Hawkes} and Hawkes and Oakes \cite{ho} as a model for earthquake appearances. Recently, they have regained a lot of interest as good models in neuroscience but also in financial econometrics, see e.g.\ Hansen et al.\ \cite{hrbr} and Chevallier \cite{julien} for the use of Hawkes processes as models of spike trains in neuroscience, see Delattre et al. \cite{mathieu} for the use of Hawkes processes in financial modeling. Finally, we refer the reader to Br\'emaud and Massouli\'e \cite{bm} for the stability properties of nonlinear Hawkes processes. 

By the form \eqref{eq:intensity} and \eqref{eq:intensity2} of the intensities, we are in a mean-field frame, that is, the intensity processes of one population depend only on the empirical measure of the other population. We will suppose that $ N  \to \infty $ such that for $k=1,2,  $ 
$$\lim_{N \to \infty} \frac{N_k}{N} \, \mbox{ exists and is in } \, ]0, 1[ .$$ 
In \cite{EvaSusanne}, we have shown that in the large population limit, when $ N \to \infty, $ self-sustained periodic behavior emerges even though each single particle does not follow periodic dynamics. In the present paper we show how this periodic behavior is also felt at a finite population size.

\subsection{An associated cascade of diffusion processes}
We represent the Hawkes processes via the associated processes 
$$
X^N_{1} (t) =\frac{1}{N_{2}} \sum_{ j = 1 }^{N_{2}}  \int_{]0, t]} h_{12 } (  t- s ) d Z^N_{2, j } ( s)  
, \; 
X^N_{2} (t) =\frac{1}{N_{1}} \sum_{ j = 1 }^{N_{1}}  \int_{]0, t]} h_{21 } (  t- s ) d Z^N_{1, j } ( s)  .
$$
Each particle belonging to the first population jumps at rate $ f_1 ( X^N_1 (t- ) ) , $ and each particle belonging to the second population at rate $f_2 ( X^N_2 (t- ) ) , $ at time $t.$ If the memory kernels $h_{12} $ and $ h_{21}$ are  exponential, then the system $ (X^N_1 (t) , X^N_2 (t))_{t \geq 0}$ is a piecewise deterministic Markov processes (PDMP). In the present paper, we do not choose exponential memory kernels, since they induce a very short memory. Instead of this, we consider {\it Erlang memory kernels} 
$$ h_{12 } (s) =  c_1 e^{ - \nu_1 s } \frac{ s^{n_1}}{n_1!}  ,  h_{21 } (s) =  c_2 e^{ - \nu_2 s } \frac{ s^{n_2}}{n_2!} ,$$
where $c_1 , c_2 \in \R , $ where $ \nu_1 , \nu_2 $ are positive constants and $ n_1 , n_2 \in \N $ the length of the delay within the memory kernel. Such kernels allow for delays in the transmission of information. In this case, the processes $ (X^N_1 (t) , X^N_2 (t))_{t \geq 0} $ alone are not Markov, but they can be completed by a {\it cascade} of processes $X^N_{k, l } , l=1, \ldots , n_k +1 , k=1, 2, $ such that this cascade is Markov. The equations defining the cascade are given by
\begin{equation}\label{eq:cascadepdmp0}
\left\{ 
\begin{array}{lcl}
d X^N_{1, l } (t) &=& [ -  \nu_{1} X^N_{1, l  } ( t) + X^N_{1, l+1} (t) ] dt , \; 1 \le l \le n_1 ,  \\
d X^N_{1, n_1+1} (t) &=& -  \nu_{1} X^N_{1,  n_1+1} (t) dt + \frac{c_1}{N_{2}} \sum_{j=1}^{N_{2} }d  Z^N_{2, j} (t) ,
\end{array}
\right.
\end{equation}
where $X^N_{1 }$ is identified with $X^N_{1, 1 } $ and where each $Z^N_{2, j } $ jumps at rate $f_{2} (  X^N_{2} (t- ) ) .$ A similar cascade describes the evolution of the second population. 
Notice that for each population the length of the cascade is related to the length of delay in the corresponding memory kernels. 

In the ``large jump intensity, small jump height''-regime, it is natural to study the {\it canonical diffusion approximation} of this cascade. It is given by the following systems of equations. The process $X^N_1 (t) $ is approached by the diffusion process $  Y^N_{1, 1 } (t)  ,$ together with its successive cascade terms,  solution of

\begin{equation}\label{eq:cascadeapprox0}
\left\{ 
\begin{array}{lcl}
d Y^N_{1, l } (t) &=& [ -  \nu_{1} Y^N_{1, l  } ( t) + Y^N_{1, l +1} (t) ] dt , \;  \quad \quad 1 \le l \le n_1 ,  \\
d Y^N_{1,  n_1 +1} (t) &=& -  \nu_{1} Y^N_{1,  n_1+1} (t) dt + c_{1} f_{2} (  Y^N_{ 2, 1} (t) ) dt  +  c_{1} \frac{\sqrt{f_2 (  Y^N_{ 2, 1} )(t)  }}{ \sqrt{ N_{2}}} d B^{2}_t .
\end{array}
\right .
\end{equation}

In the above system, $B^2 $ is a one dimensional standard Brownian motion which is associated to the jump noise of the second population and appears only in the last term of the cascade. Notice also that only the last term of the cascade encodes the interactions with the second population, through the jump rate function $f_2 $ and the jump intensity $ f_2 ( Y^N_{2, 1 } ( t) )  $ of the second population.  The above system \eqref{eq:cascadeapprox0} has to be completed by a similar cascade of length $n_2 +1 $ describing the jump intensity of the second population. This diffusion approximation is a good approximation of the original cascade of PDMP's, and the weak approximation error $| E ( \varphi ( X_t^N)) - E ( \varphi ( Y_t^N)) | , t \le T,$ is of order $ T N^{-2}, $ for sufficiently smooth test functions $\varphi $ (see \cite{EvaSusanne}). 

The present paper is devoted to the study of the long time behavior of this diffusion approximation $ Y^N $ and its large deviation properties. 

Let us start by discussing the main features of this diffusion process. Firstly, we have to treat the memory terms -- the terms following the first line of the above cascade --  as auxiliary variables. This gives rise to coordinates of $ Y^N$ without noise. 
Therefore we obtain a degenerate high-dimensional diffusion process $Y^N $ driven by two-dimensional Brownian motion. This diffusion turns out to be hypo-elliptic; indeed, it is easy to check that the weak H\"ormander condition is satisfied. The drift of the diffusion is almost linear -- only the two coordinates encoding the interactions between the two populations do not have a linear drift term. 

The interactions are transported through the system according to a ``chain of reactions'', i.e.\ the drift of a given coordinate does only depend on the coordinate itself and the next one. We call this the {\it cascade structure of the drift vector field}. This structure enables us to use results on the control properties of the diffusion \eqref{eq:cascadeapprox0} obtained by Delarue and Menozzi \cite{delaruemenozzi} in a recent paper establishing density estimates for such chains of differential equations. Due to this structure, the coordinates of the diffusion do not travel at the same speed. Indeed, the coordinate $ Y^N_{1, n_1 +1} (t) ,$ driven by Brownian motion, evolves at speed $t^{1/2} , $ the coordinate $Y^N_{1, n_1} (t)  $ at speed $ t^{1+ 1/2 } $ and more generally, $ Y^N_{1, n_1 - l } (t) $ at speed $t^{l +1 + 1/2  }.$ In particular, over small time intervals $ [0, \delta ] $ and for all coordinates which are not driven by Brownian motion, the drift does play a crucial role in the control problem of our diffusion, and this is reflected in the cost associated to the control (see \cite{delaruemenozzi} and the proof of Theorem \ref{theo:6} below).     

Cascades or  chains of reactions similar to the one described in \eqref{eq:cascadeapprox0} appear also in systems of coupled oscillators in models of heat conduction where the first oscillator is forced by random noise. Rey-Bellet and Thomas \cite{reybellet} have studied the large deviation properties of such systems, and parts of our proofs are inspired by their approach.

\subsection{Monotone cyclic feedback systems}
The deterministic part of the system \eqref{eq:cascadeapprox0} is given by an $ n_1 + n_2 + 2-$dimensional dynamical system $ (x_{k, l} ( t)) , 1 \le l \le n_k+1, k = 1 , 2,$ which is solution of 
\begin{equation}\label{eq:ll}
 \frac{d x_{1,l} (t)  }{d t } = - \nu_1 x_{1, l } ( t)  + x_{1, l+1} (t)  , 1 \le l \le n_1,  \frac{d x_{1, n_1+1} (t) }{d t } = - \nu_1 x_{1, n_1+1} (t)  + c_1 f_2 ( x_{2,  1} (t) ) ,
\end{equation} 
together with the chain of equations describing the second population. This system is a {\it monotone cyclic feedback system} in the sense of Mallet-Paret and Smith \cite{malletparet-smith}. The most important point is that the long time behavior of \eqref{eq:ll}, i.e.\ the structure of its $\omega-$limit sets, is well-understood. More precisely, there exist explicit conditions ensuring the existence of a single linearly unstable equilibrium point $x^* $ of this limit system, together with a finite number of periodic orbits such that at least one of them is asymptotically orbitally stable (see Theorem \ref{theo:orbit} below). This result goes back to deep theorems in dynamical system's theory, obtained by Mallet-Paret and Smith \cite{malletparet-smith} and used in a different context in Bena\"{\i}m and Hirsch \cite{michel}, relying on the Poincar\'e-Bendixson theorem. 

In other words, there exist $ x_1, \ldots , x_M \in \R^{ n_1 + n_2 + 2 }   $ such that the solutions $\Gamma_1 (t) , \ldots, \Gamma_M (t)$ of \eqref{eq:ll} issued from these points are non-constant periodic trajectories, i.e., they are cycles (or periodic orbits). At least one of these cycles is an attractor of \eqref{eq:ll}, which means that the other solutions of \eqref{eq:ll} will converge  to this limit cycle in the long run (provided they start within the domain of attraction of this limit cycle). The limit cycles encode oscillatory behavior of the system; that is, periods where the first population has large jump intensity, while the jump intensity of the second population is small, are followed by periods where the second population has large jump intensity, but not the first one. This has been supported by simulations provided in \cite{EvaSusanne}. 

Due to the presence of noise, the diffusion $Y^N$ may switch from one limit cycle to another. But for large $N , $ $Y^N$ will tend to stay within tubes around the limit cycles $ \Gamma_1, \ldots, \Gamma_M $ during long periods, before eventually leaving such a tube after a time which is of order $ e^{ N \bar V },$ where $\bar V$ is related to the cost of steering the process from the cycle to the boundary of the tube (see Proposition \ref{prop:sigma_0first} and \ref{prop:sigma_0second} below). As time goes by, the diffusion will therefore spend very long time intervals in vicinities of one of the limit cycles -- interrupted by short lasting excursions into the rest of the state space. It is therefore natural to consider the concentration of the invariant measure $\mu^N $ of $Y^N$ around the periodic orbits -- if this invariant measure exists and is unique.   

It is not difficult to show that, for fixed $N, $ the process possesses a unique invariant probability measure $ \mu^N .$  Moreover, a Lyapunov type argument implies that the process converges to its invariant regime at exponential speed. For fixed $N, $ $\mu^N $ is of full support but its mass is concentrated around the periodic orbits of the limit system \eqref{eq:ll}. More precisely, we can show that for any open set $ D$ with compact closure and smooth boundary, 

\begin{equation}\label{eq:ld}
 \mu^N ( D) \sim C e^{ -  [\inf_{ x \in D} W (x)]  N} ,
\end{equation}

where the cost function $W(x) $ is related to the control properties of system \eqref{eq:cascadeapprox0} and is given explicitly in \eqref{eq:W} below.

In order to prove this result, we rely on the approach of Freidlin and Wentzell \cite{FW} to sample path large deviations of diffusions, developed further in Dembo and Zeitouni \cite{DZ}. Both \cite{FW} and \cite{DZ} suppose that the underlying diffusion is elliptic -- which is not the case in our situation. Recently, Rey-Bellet and Thomas \cite{reybellet} have extended the results of Freidlin and Wentzell \cite{FW} to degenerate diffusions, and our proof is inspired by their paper. The most important point of our paper is to establish the necessary control theory in our framework. For this, an important tool are recent results obtained by Delarue and Menozzi \cite{delaruemenozzi}. Moreover, since we are dealing with periodic orbits rather than with equilibrium points, we have to extend the notion of {\it small time local controllability} to the situation where the drift vector field does play a role in the sense of a shift on the orbit, see Theorem \ref{theo:stlc} below.   

This paper is organized as follows. In Section \ref{sec:2} we state the main assumptions and provide a short study of the limit system together with its $ \omega-$limit set in Theorem \ref{theo:orbit}. In Section \ref{sec:3}, we state the main results of the paper which are the positive Harris recurrence of $Y^N$ in Theorem \ref{theo:harris} together with the large deviation properties of the invariant measure $\mu^N$ of the diffusion as $ N \to \infty , $ in Theorem \ref{theo:main}. Section \ref{sec:control} provides a proof of the Harris recurrence of $Y^N, $ based on the control theorem. Finally, Section \ref{sec:5} is devoted to a study of the control properties of the process. Here, we first show that the process is strongly completely controllable. The proof of this fact relies on the prescription of a control that allows to decouple the two populations and to make use of the linear structure of the (main part of the) drift. We also study the continuity properties of the cost functional -- a study which is not trivial in the present frame of strong degeneracy of the diffusion matrix. Section \ref{sec:6} gives the proof of Theorem \ref{theo:main}.

\section{Main assumptions and results}\label{sec:2}
In what follows,  we use the notations introduced above. Moreover, for fixed $ n \geq 1 , $ elements $x $ of $ \R^n $ shall be denoted by $ x = ( x_1 , \ldots , x_n  ) ,$ and $ \R^n $ will be endowed with the Euclidean norm denoted by $ \| x\| .$ Finally,  for matrices $ A \in \R^{n \times n } ,$ $ \| A \| $ denotes the associated operator norm.       

Our first main assumption is the following.

\begin{ass}\label{ass:1}
(i) $f_1$ and $f_2: \R \to  \R_+  $ are bounded analytic functions which are strictly lower bounded, i.e., $ f_1  (x) , f_2 (x) \geq \underbar f > 0 $ for all $x \in \R .$ Moreover, $f_1 $ and $f_2 $ are non decreasing.\\ 
(ii) There exists a finite constant $L $ such that
for every $x$ and $x' $ in $\R,$ 
\begin{equation}
\label{Lipsch-f}
|f_1 (x)  -f_1 (x') |   + |f_2 (x)  -f_2 (x')    |  \le L |x-x'|.
\end{equation}\\
(iii) The functions $ h_{12}, h_{21} $ are given by 
\begin{equation}\label{eq:erlang}
 h_{12 } (s) =  c_1 e^{ - \nu_1 s } \frac{ s^{n_1}}{n_1!}  ,  h_{21 } (s) =  c_2 e^{ - \nu_2 s } \frac{ s^{n_2}}{n_2!} ,
\end{equation}
where $n_1, n_2 \in \N_0 , c_1, c_2 \in \{-1,1\} $ and $ \nu_1 , \nu_2 > 0$ are fixed constants. 
\end{ass}

Under the above assumption, it is standard to show that the Hawkes process with the prescribed dynamics above exists.

\begin{prop}[Prop.1 of \cite{EvaSusanne}]
Under Assumption \ref{ass:1} there exists a path-wise unique Hawkes process \\$ (Z^{N}_{k,i} (t)_{ 1 \le k \le 2, 1 \le i \le N_k } )$ with intensity \eqref{eq:intensity}, for all $ t \geq 0.$  
\end{prop}

\subsection{An associated cascade of piecewise deterministic Markov processes (PDMP's)}\label{sec:erlang}
In the sequel we establish a link between the Hawkes process $ (Z^{N}_{k,i} (t)_{1 \le k \le 2, 1 \le i \le N_k} )$ -- which is of infinite memory -- and an associated system of Markov processes.  This relation exists thanks to the very specific structure of the memory kernels $h_{12}, h_{21}  $ in \eqref{eq:erlang}. Such kernels are called {\it Erlang memory} kernels;  they can describe delays in the transmission of information.  In \eqref{eq:erlang}, $n_1+1$ is the order of the delay, i.e., the number of differential equations needed for population $1$ to obtain a system without delay terms, and $ n_2 +1 $ is the order of delay for population $2.$ The delay of the influence e.g.\ of population $2$ on population $1$ is distributed and takes its maximum absolute value at $n_2/\nu_2$ time units back in time, and the mean is $(n_2+1)/\nu_2$ (if normalizing to a probability density). The higher the order of the delay, the more concentrated is the delay around its mean value, and in the limit of $n_2 \rightarrow \infty$ while keeping $(n_2+1)/\nu_2$ fixed, the delay converges to a discrete delay. The sign of $c_1$ and $c_2$ indicates if the influence is inhibitory or excitatory.

We introduce the family of adapted c\`adl\`ag processes 
\begin{equation}\label{eq:intensity21}
X^N_{1} (t) :=\frac{1}{N_{2}} \sum_{ j = 1 }^{N_{2}}  \int_{]0, t]} h_{12 } (  t- s ) d Z^N_{2, j } ( s)  = \int_{]0, t]} h_{12 } (t-s) d \bar Z^N_{2} ( s) 
\end{equation}
and 
\begin{equation}\label{eq:intensity22}
X^N_{2} (t) :=\frac{1}{N_{1}} \sum_{ j = 1 }^{N_{1}}  \int_{]0, t]} h_{21 } (  t- s ) d Z^N_{1, j } ( s)  = \int_{]0, t]} h_{21 } (t-s) d \bar Z^N_{1} ( s) ,
\end{equation}
where $ \bar Z^N_{k} ( s) = \frac{1}{N_{k}} \sum_{j=1}^{N_{k}} Z^N_{k, j } ( s) , k=1, 2  .$ Recalling \eqref{eq:intensity},  it is clear that the dynamics of the system is entirely determined by the dynamics of the processes $ X^N_{k } ( t- ) , t \geq 0 .$ Indeed, any particle belonging to the first  population jumps at rate $ f_1 (X^N_{1} (t-)), $ and any particle belonging to the second population at rate $ f_2 ( X^N_{2} (t-) ) .$     Without assuming the memory kernels to be Erlang kernels, the system $(X^N_{k } , 1 \le k \le 2  ) $ is not Markovian: For general memory kernels, Hawkes processes are truly infinite memory processes. 

When the kernels are Erlang, given by \eqref{eq:erlang},
taking formal derivatives in \eqref{eq:intensity21} and \eqref{eq:intensity22} with respect to time $t$ and introducing for any $ k =1,2$ and $ 1  \le l \le n_{k } +1 $  
\begin{equation}
 X^N_{k, l} ( t) :=  c_{k} \, \int_{]0, t]}   \frac{ (t-s)^{n_k- (l-1)}}{(n_k- (l-1))! }  e^{- \nu_{k} ( t-s)} d \bar Z^N_{k+1} ( s)  ,
\end{equation} 
where we identify population $ 2+1$ with population $1$, we obtain the following system of stochastic differential equations driven by Poisson random measure.

\begin{equation}\label{eq:cascadepdmp}
\left\{ 
\begin{array}{lcl}
d X^N_{k, l } (t) &=& [ -  \nu_{k} X^N_{k, l  } ( t) + X^N_{k, l+1} (t) ] dt , \; 1 \le l \le n_k ,  \\
d X^N_{k, n_k+1} (t) &=& -  \nu_{k} X^N_{k,  n_k+1} (t) dt + c_{k} d \bar Z^N_{k+1} (t) ,
\end{array}
\right. 
\end{equation}
$k=1, 2.$ 

Here, $X^N_{k }$ is identified with $X^N_{k, 1 }, $ and each $Z^N_{k, j } $ jumps at rate $f_{k} (  X^N_{k, 1 } (t- ) ) .$ We call the system \eqref{eq:cascadepdmp} a {\em cascade of memory terms}. 
Thus, the dynamics of the Hawkes process $ (Z^N_{k, i } (t))_{  1 \le k \le  2, 1 \le i \le N_k}$ is entirely determined by the PDMP $ ( X^N_{k, l })_{(1\le k \le 2, 1 \le l \le n_k +1 )}  $ of dimension $n := n_1 + n_2 + 2 .$ 
 
\subsection{A diffusion approximation in the large population regime}
In the large population limit, i.e.\ when $ N \to \infty , $ it is natural to consider the diffusion process approximating the above cascade of PDMP's. This diffusion approximation is given by 

\begin{equation}\label{eq:cascadeapprox}
\left\{ 
\begin{array}{lcl}
d Y^N_{k, l } (t) &=& [ -  \nu_{k} Y^N_{k, l  } ( t) + Y^N_{k, l +1} (t) ] dt , \;  \quad \quad 1 \le l \le n_k ,  \\
d Y^N_{k,  n_k +1} (t) &=& -  \nu_{k} Y^N_{k,  n_k+1} (t) dt + c_{k} f_{k+1} (  Y^N_{ k+1, 1} (t) ) dt  +  c_{k} \frac{\sqrt{f_{k+1} (  Y^N_{ k+1, 1} )(t)  }}{ \sqrt{ N_{k+1}}} d B^{ k+1}_t ,
\end{array}
\right .
\end{equation}

$k=1, 2,$ where population $2+1$ is identified with $1$ and where $B^{1} , B^2 $ are independent standard Brownian motions (compare to Theorem 4 of \cite{EvaSusanne}). The diffusion $Y^N = (Y^N_t)_{t \in \R_+} $ takes values in $ \R^n $ with $ n = n_1 + n_2 +2 .$ 

By Theorem 4 of \cite{EvaSusanne}, we know that $Y^N$ is a good approximation of the PDMP $X^N $ since the weak approximation error can be controlled by 
$ \sup_x | E_x ( \varphi ( X^N (t) ))-   E_x ( \varphi ( Y^N (t) )) | \le C ( \varphi)  T N^{-2 } ,$
for all $ t \le T , $ for sufficiently smooth test functions $ \varphi .$ We therefore concentrate on the study of this diffusion process $Y^N.$  

We write $A^N$ for the infinitesimal generator of the process \eqref{eq:cascadeapprox}. Moreover, we denote by $ Q_x^{N}$ the law of the solution $ (Y^N(t), t \geq 0) $ of  \eqref{eq:cascadeapprox}, starting from $Y^N (0) = x ,$ for some $x \in \R^n , $ and by $ E_x^N$ the corresponding expectation. 

We study the above diffusion when  $ N_1 , N_2 \to \infty $ such that $ N_1/N =: p_1 $ and $ N_2/N =: p_2$ remain constant.  Re-numbering the coordinates of $ Y^N $ as $ (Y^N_1, \ldots, Y^N_n ), $ where $ n = n_1 + n+2 + 2, $ we may introduce 
\begin{equation}\label{eq:drift}
b(x) := \left( 
\begin{array}{c} 
- \nu_1 x_{ 1 } + x_{2 } \\
- \nu_1 x_{2} + x_{3} \\
\vdots\\
- \nu_1 x_{n_1 +1 } + c_1 f_2 ( x_{n_1+2} ) \\
- \nu_2 x_{n_1+2 } + x_{n_1+3 } \\
\vdots \\
- \nu_2 x_{n} + c_2 f_1 ( x_{1} ) 
\end{array}
\right) , \; 
\sigma (x) := \left( 
\begin{array}{cc}
0&0\\
\vdots & \vdots \\
0 & \frac{c_1}{\sqrt{p_2}} \sqrt{ f_2 ( x_{n_1+2 } ) } \\
0 & 0 \\
\vdots & \vdots \\
\frac{c_2}{\sqrt{ p_1} } \sqrt{ f_1 ( x_{1 } ) } & 0 
\end{array}
\right) ,
\end{equation}
which are the drift vector of \eqref{eq:cascadeapprox} and the associated diffusion matrix which is an $ n \times 2 -$ matrix. Notice that $ \sigma $ is highly degenerate; there is a two-dimensional Brownian motion driving an $n-$dimensional system. We may rewrite \eqref{eq:cascadeapprox} as 

\begin{equation}\label{eq:diffusionsmallnoise}
 d Y_t^N   = b ( Y_t^N  ) dt + \frac{1}{ \sqrt{N}} \sigma ( Y_t^N  ) d B_t , 
\end{equation} 
with $ B_t = ( B^1_t , B^2_t ) .$  
 
The aim of this paper is to study this diffusion $Y^N $ and its long time behavior in the large population limit (i.e.\ as the noise term tends to $0$). We will show that this diffusion presents oscillations in the long run and we will study the large population limit of the associated invariant measure. This will be done relying on the Freidlin-Wentzell theory (see \cite{FW} and \cite{DZ}) on sample path large deviations for diffusion processes which has been extended recently to the case of (some) degenerate diffusions in Rey-Bellet and Thomas \cite{reybellet}.   

We start with a discussion of the deterministic limit system associated to \eqref{eq:diffusionsmallnoise}. 

\subsection{Monotone cyclic feedback systems}
Consider the solution of 
\begin{equation}\label{eq:limit}
\dot x (t) = b (x (t) )  , 
\end{equation}
i.e.\ of 
\begin{eqnarray}\label{eq:cascade}
\frac{d x_{i} (t) }{dt} &=& - \nu_1 x_i  (t)  +  x_{i+1} (t)  , \, \, 1 \le i  \le  n_1   , \; 
\frac{d x_{n_1 +1 } (t) }{dt} = - \nu_1 x_{n_1 + 1 } (t)   + c_1  f_{2} ( x_{n_1+2 } (t)) , \nonumber \\
\frac{d x_{i} (t) }{dt} &=& - \nu_2 x_{i} (t) +  x_{i+1} (t) , \, \, n_1+2  \le i  < n    , \; \frac{d x_{n } (t) }{dt} = - \nu_2 x_{n } (t)  + c_2  f_{1} ( x_{1 } (t)) . 
\end{eqnarray}
This system is a monotone cyclic feedback system as considered e.g.\ in \cite{malletparet-smith} or as  in (33) and (34) of \cite{michel}. If $c_1 c_2  > 0, $ then the system \eqref{eq:cascade} is of total positive feedback, otherwise it is of negative feedback. It can be shown easily (see Prop.\ 5 of \cite{EvaSusanne}) that \eqref{eq:cascade} admits a unique equilibrium $x^*$ if  $ c_1c_2  < 0 .$ 

We now present special cases where system \eqref{eq:cascade} is necessarily attracted to non-equilibrium periodic orbits. Recall that  $ n =  n_1+ n_2  +2$  is the dimension of \eqref{eq:cascade}.  

The following theorem is based on Theorem 4.3 of \cite{malletparet-smith} and generalizes the result obtained in Theorem 6.3 of \cite{michel}. We quote it from \cite{EvaSusanne}. 

\begin{theo}\label{theo:orbit}[Theorem 3 of \cite{EvaSusanne}]

Grant Assumption \ref{ass:1}.  Put $\varrho := c_1 c_2  f_1' (x^*_{
1})  f_2' (x^*_{
n_1+2 })  $ and suppose that $ \varrho < 0.$ Consider all solutions $\lambda
$ of
\begin{equation}\label{eq:racinesunite}
(\nu_1 + \lambda)^{n_1 +1} \cdot (\nu_2 + \lambda)^{n_2
+1} = \varrho
\end{equation}
and suppose that there exist at least two solutions $\lambda$ of
\eqref{eq:racinesunite} such that
\begin{equation}\label{eq:unstable}
\mbox{Re } ( \lambda ) > 0.
\end{equation}
Then
$x^* $ is linearly unstable, and the system \eqref{eq:cascade}
possesses at least one, but no more than a finite number of non constant periodic
orbits. Any $ \omega-$limit set is either the equilibrium $ x^* $ or one of these periodic orbits. At least one of the periodic orbits  is orbitally asymptotically stable. 
\end{theo}

Notice that $\varrho < 0 $ implies that $ c_1 c_2 < 0, $ that is, we are in the frame of a total negative feedback. In the sequel, we shall always assume that the assumptions of Theorem \ref{theo:orbit} are satisfied and we introduce 

\begin{ass}\label{ass:4}
We suppose that  $\rho := c_1 c_2  f_1' (x^*_{
1})  f_2' (x^*_{
n_1+2 }) $ satisfies that $ \rho < 0$ and that there exist at least two solutions $\lambda$ of
\eqref{eq:racinesunite} with $
\mbox{Re } ( \lambda ) > 0.$ 
\end{ass}
Under Assumption \ref{ass:4}, there exists a finite number of periodic orbits, and we write $ K_1 = \{ x^* \} $ for the unstable equilibrium point and  $K_2 , \ldots , K_L $ for the periodic orbits of the limit system \eqref{eq:cascade}. Moreover, we write $K =  \bigcup_{l=1}^L K_l $ and 
\begin{equation}\label{eq:b}
 B_\varepsilon (K) = \{ x \in \R^n : dist ( x, K) < \varepsilon \} ,
\end{equation} 
where $dist ( x, K) = \inf \{ \| x - y \| ,  y \in K \} $ and where $ \| \cdot \| $ is the Euclidean norm on $\R^n.$ 

Due to the presence of noise, the diffusion $Y^N$ will be able to switch from the vicinity of one periodic orbit to the vicinity of another orbit. However, as $N \to \infty, $ the diffusion will stay within tubes around periodic orbits during longer and longer periods, before eventually leaving such a tube after a time which is of order $ e^{ N \bar V },$ where $\bar V$ is related to the cost of steering the process from the orbit to the boundary of the tube. This behavior can be read on the invariant measure of $Y^N, $ 
and the main result of this paper is to show that the invariant measure of the diffusion will concentrate around the stable periodic orbits of \eqref{eq:cascade}  as $ N \to \infty .$

\section{Main results : Large deviations for the diffusion approximation $Y^N$ }\label{sec:3}
We start with some preliminary results on the diffusion process $Y^N .$ 

\begin{theo}\label{theo:harris}
Grant Assumption \ref{ass:1}. Then $Y^N$ is positive Harris recurrent with unique invariant probability measure $ \mu^N .$ The invariant measure $ \mu^N $ is of full support.  
\end{theo} 

The proof of this result will be given in Section \ref{sec:control} below. It is based on two main ingredients. The first ingredient is the existence of a Lyapunov function, a result that has been obtained in \cite{EvaSusanne} and that we quote from there. In order to state this result, introduce
$\tau_\varepsilon = \inf \{ t \geq 0 : Y_t^N \in B_\varepsilon (K) \} ,$ where $ B_\varepsilon (K) $ has been defined in \eqref{eq:b}. 
 
\begin{prop}[Prop.\ 5 and Theorem 5 of \cite{EvaSusanne}]\label{prop:lyapunov}
There exists a function $ G : \R^n \to \R_+ $ not depending on $N,$ such that $ \lim_{ |x | \to \infty } G( x) = \infty ,$ and constants $a ,b > 0 $ not depending on $N$ such that $ A^N G \le - a G + b .$
Moreover, we also have
$$ E^N_x \tau_\varepsilon \le c G(x)  ,$$
for some constant $c > 0$ not depending on $N.$ 
\end{prop}

The second main ingredient to prove the Harris recurrence is the following:  Despite the fact that $Y^N$ is highly degenerate, the weak H\"ormander condition is satisfied on the whole state space, as it has been shown in Proposition 7 of \cite{EvaSusanne}. 

Once the Harris recurrence of the process is proven, we turn to the large deviation properties of $Y^N.$ We firstly introduce the cost functional related to the control problem of the diffusion $Y^N.$ For that sake, for some time horizon $t_1<\infty$ which is arbitrary but fixed, write $\,\tt H\,$ for the Cameron-Martin space of measurable functions ${h}:[0,t_1]\to \R^2 $ having absolutely continuous components ${ h}^\ell(t) = \int_0^t \dot h^\ell(s) ds$ with $\int_0^{t_1}[{\dot h}^\ell]^2(s) ds < \infty$, $1\le \ell\le 2$. For $ h \in \,\tt H\, , $ we put $ \| \dot h\|_\infty := \| \dot h^1 \|_\infty + \| \dot h ^2 \|_\infty .$  For $x\in \R^n $ and ${ h}\in{\tt H}$, consider the deterministic system 

\begin{equation}\label{eq:generalcontrolsystem}
\varphi = \varphi^{( {h}, x)} \; \mbox{solution to}\; d \varphi (t) = b (  \varphi (t) ) dt +  \sigma( \varphi (t) ) \dot h(t) dt, \; \mbox{with $\varphi (0)=x,$}  
\end{equation}

on $[0, t_1 ].$  As in Dembo and Zeitouni \cite{DZ}, we introduce the rate function $ I_{x, t_1} (f) $ on $ C ( [0, t_1 ], \R^n ) $ by 
\begin{equation}\label{eq:rate}
I_{x, t_1} (f) = \inf_{ {h} \in {\tt H} : \varphi^{( { h}, x)}  (t) = f(t) , \; \forall t \le t_1 } \frac12 \int_0^{t_1} [| \dot{ h}^1(s)|^2+ | \dot{ h}^2 (s) |^2] ds  ,
\end{equation}
where $ \inf \emptyset = + \infty . $ Notice that the above rate function is not explicit since the diffusion matrix $\sigma $ is degenerate. 

We then introduce the cost function $ V_t (x, y )$ which is given by 
$$ V_t ( x, y) = \inf_{ { h } \in {\tt H} : \varphi^{ ( {h} , x)} (t) = y } \frac12 \int_0^t [| \dot{ h}^1(s)|^2+ | \dot{ h}^2 (s) |^2] ds , \;  V( x, y ) = \inf_{t > 0 } V_t ( x,y).$$
Finally, for any two sets $ B, C \in {\mathcal B} ( \R^n ) $ we define
$$ V( B, C ) = \inf_{x \in B, y \in C } V( x, y ) .$$ 
As in Freidlin and Wentzell \cite{FW} we say that two points $x$ and $ y $ are equivalent and we write $ x \sim y , $ if and only if $V( x, y ) = V(y, x ) = 0 . $ Notice that the $\omega -$limit set $K= \{ x^*\}  \cup \bigcup_{l=2}^L K_l $ consists of $L$ such equivalence classes with respect to this equivalence relation. In \cite{FW}, Chapter 6.3, Freidlin and Wentzell introduce graphs on the set $ \{1 , \ldots , L \} $ in the following way. For any fixed $ i \le L , $ an $ \{i \}-$graph is a set consisting of arrows $ m \to n $ where all starting points $m$ of such an arrow are $\neq i ,$  such that every $ m \neq i $ is the initial point of exactly one arrow and such that there are no closed cycles in the graph. Intuitively, such an $\{i\}-$graph describes the possible ways of going from some $ K_m , $ $m \neq i , $ to $K_i, $ following a path $K_m =: K_{m_1}  \rightarrow K_{m_2} \rightarrow K_{m_3} \rightarrow \ldots \rightarrow K_i, $ without hitting one of the sets $K_{m_j} $ twice. Therefore, $ \{i\}-$graphs describe all possible ways of passages between the sets $ K_j ,$ ending up in $K_i.$ An alternative description of such passages is given by means of the hierarchy of $L-$cycles, as pointed out in Freidlin and Wentzell \cite{FW}, Chapter 6.6. Writing $ G \{ i \} $ for the set of all possible $ \{i\}-$graphs, we then introduce 
\begin{equation}
W( K_i ) = \min_{g \in G\{i\} } \sum_{m \to n \in g } V( K_m, K_n ) ,
\end{equation}
which is the minimal cost of going from any $ K_j $ to $K_i $ for some $ j \neq i .$ 

The following theorem is our main result.  

\begin{theo}\label{theo:main}
Grant Assumptions \ref{ass:1} and \ref{ass:4}. Then for any open set $ D$ with compact closure and smooth boundary satisfying $ dist ( D, K ) > 0, $ we have 
\begin{equation}\label{eq:main}
 \lim_{N \to \infty } \frac1N \log \mu^N (D) = - \inf_{ x \in D} W(x) ,
\end{equation}
where 
\begin{equation}\label{eq:W}
 W(x) = \min_i \left( W( K_i ) + V ( K_i , x) \right) - \min_j W( K_j ) .
\end{equation}
\end{theo}

A similar result has been established by Rey-Bellet and Thomas in a recent paper on the asymptotic behavior of thermal non-equilibrium steady states in driven chains of anharmonic oscillators, which is a model of heat conduction, see \cite{reybellet}. Our proof is inspired by their approach. The main difference with respect to their paper is the fact that the $\omega -$limit set of our model is built of periodic orbits rather than stable equilibrium points. As a consequence, the action of the drift vector field plays an important role close to points of any of the periodic orbits. This implies that the property of {\it small time local controllability} -- essential for the proof -- has to be adapted to the present situation. Moreover, the controllability of our system has to be carefully studied  -- indeed we are facing a degenerate situation where Brownian motion is only present in two coordinates of a (possibly) high-dimensional system. As a consequence, the controllability of the system as well as the continuity of $ V (x, y) $ with respect to $x$ and $y$ are difficult questions. It is the {\it cascade structure} of the drift vector which is crucial for our purpose -- we will come back to this point later. 

\section{Proof of Theorem \ref{theo:harris}}\label{sec:control}

We will use the control theorem which goes back to Stroock and Varadhan \cite{StrVar-72}, see also Millet and Sanz-Sol\'e \cite{MilSan-94}, theorem 3.5, in order to prove Theorem \ref{theo:harris}. The following proposition summarizes the inclusion of the control theorem which is important for our purpose. 

\begin{prop}\label{theo:4bis}
Grant Assumption \ref{ass:1}. Denote by $ Q_x^{N,t_0} $ the law of the solution $ (Y^N(t))_{0 \le t \le t_0} $ of  \eqref{eq:cascadeapprox}, starting from $Y^N (0) = x .$ Let $\,\varphi = \varphi^{(N, {h},x)}\,$ denote a solution to

\begin{equation}\label{eq:withN}
d \varphi (t) \;=\; b ( \varphi (t) )\, dt \;+\; \frac{1}{\sqrt{N}} \sigma ( \varphi  (t) )\, \dot{h}(t)\, dt   \quad,\quad \varphi (0)=x.
\end{equation}

Fix $ x   $ and  $  { h } \in \tt H $ such that $\,\varphi = \varphi^{(N, {h},x)}\,$ exists on some time interval $ [ 0, \widetilde T ] $ for $\widetilde T > t_0.$ Then
$$ \left(\varphi^{(N, { h} , x ) }\right)_{| [0,  t_0 ] } \in \overline{ {\rm supp} \left( Q_x^{N,t_0 } \right)}.$$
\end{prop}
 
We now show how to use Proposition \ref{theo:4bis} in order to prove Theorem \ref{theo:harris}.  

{\it Proof of Theorem \ref{theo:harris}. }

By Proposition \ref{prop:lyapunov}, putting $F := \{ x : G(x)  \le 2 b /a \} ,$ the set $F$ is visited infinitely often by the process $Y^N , $ almost surely. 
Fix $ x \in F $ and recall that $ x^* $ is the unique equilibrium point of the system \eqref{eq:cascade}. We will show in Theorem \ref{theo:control} below that it is possible to choose ${ h } \in \tt H $ such that $\,\varphi = \varphi^{(N, {h},x)}\,$ satisfies $ \varphi (T) = x^* $ (for some arbitrary fixed $ T > 0 $). 

We have therefore shown the following assertions.

\begin{enumerate}
\item 
There exists an attainable point $x^* $ for $ Y^N .$ 
\item
There exists a Lyapunov function for $Y^N, $ in the sense of Proposition \ref{prop:lyapunov}.
\item
The weak H\"ormander condition holds. 
\end{enumerate} 

Under these conditions, it is classical to show (see e.g.\ Theorem 1 of H\"opfner et al.\ \cite{hh3}), that $Y^N$ is positively recurrent in the sense of Harris. The fact that $\mu^N $ is of full support follows again from Theorem \ref{theo:control} below implying that the control system \eqref{eq:withN} is strongly completely controllable. This concludes the proof. 
\hfill $\bullet$ 

In the following we will prove that the control system \eqref{eq:generalcontrolsystem} is {\it strongly completely controllable}.

\section{Controllability}\label{sec:5} 
\begin{theo}\label{theo:control}
Grant Assumption \ref{ass:1}.  Then the control system $\,\varphi = \varphi^{({ h},x)}\,$ given by
$$
d \varphi (t) \;=\; b ( \varphi (t) )\, dt \;+\; \sigma ( \varphi  (t) )\, \dot{h}(t)\, dt   \quad,\quad \varphi (0)=x, 
$$
is strongly completely controllable, i.e.\ for all $T > 0 ,$ for any pair of points $ x, y \in \R^n ,$  there exists a control $ {h } \in \,\tt H\,  $ such that $\varphi^{(h,x)}  (T) = y .$ 
\end{theo}

\begin{proof}
The main idea of the proof is to use the fact that the drift vector field is linear -- except for the last coordinate of each population encoding the interactions between the two populations. Imposing a trajectory for the two coordinates carrying the noise -- and carrying the interactions -- allows to decouple the two populations and to rely on linear control problems. Our Ansatz is to write $ \varphi (t) = ( \Phi (t) , \Psi (t) ) $ with $ \Phi (t ) = (\varphi_1(t) , \ldots , \varphi_{n_1 +1 } (t) ) ,$ $ \Psi (t) = (\varphi_{n_1+2 } (t), \ldots, \varphi_n (t) ) .$ $\Phi ( t) $ summarizes the coordinates describing the first population of particles, $ \Psi (t) $ describes the second population. We choose $ \Phi $ and $ \Psi$ such that they are solution of a different control system, given by

\begin{equation}\label{eq:newcontrol}
\dot \Phi (t) = F_1 ( \Phi ( t))  + B_1 u^1 ( t) ,
\end{equation}

where $ B_1 \in \R^{n_1+1}$ is the vector given by 
\begin{equation}\label{eq:B}
B_1 = \left( \begin{array}{c}
0 \\
0 \\
\vdots \\
0 \\
1 
\end{array}
\right) 
\end{equation}
and where for $x = (x_1, \ldots , x_{n_1 +1} ),$ 
\begin{equation}\label{eq:F}
 F_1 (x)  = \left( 
\begin{array}{c }
- \nu_1 x_1 + x_2 \\
- \nu_1 x_2+ x_3 \\
\vdots \\
- \nu_1 x_{n_1 } + x_{n_1 +1} \\
- \nu_1 x_{n_1 +1} 
\end{array} 
\right) .
\end{equation}
Analogous definitions apply to the second population described by $ \Psi  (t) .$ 

In what follows, by abuse of notation, we will systematically write $ x= (x_1, \ldots , x_{n_1 +1} ) $ for starting configurations of $ \Phi ( t) $ or $ x = ( x_{n_1 + 2} , \ldots , x_n ) $ for those of $\Psi (t) $ or $ x = (x_1 , \ldots , x_n ) $ for 
starting configurations of the entire system, depending on the context.

Notice that writing $ A_1 := \left( \frac{\partial F_1^i (x) }{\partial x_j } \right)_{ 1 \le i, j \le n_1 +1 } ,$ we can rewrite \eqref{eq:newcontrol} as
\begin{equation}\label{eq:newcontrolsecond}
 \dot \Phi (t) = A_1 \Phi ( t) + B_1 u^1 (t) .
\end{equation} 
By Theorem 1.11 in Chapter 1 of Coron \cite{coron}, the problem \eqref{eq:newcontrolsecond} is controllable at time $T$ if and only if the associated Gram matrix
$$ Q_T := \int_0^T e^{(T- t)A_1} B_1 B_1^* (e^{(T- t ) A_1 })^* dt $$
is invertible. But by the cascade structure of the drift, $ B_1, A_1 B_1, A_1^2 B_1 , \ldots , A_1^{n_1} B_1$ span $ \R^{n_1 +1 } , $ implying that $Q_T $ is non degenerate. 
As a consequence, for any $ x= (x_1 , \ldots, x_{n_1 +1 } )$ and $ y = ( y_1 , \ldots , y_{n_1 +1}) $ in $ \R^{n_1 + 1 } $ there exists a control $ u^1 (t) $ steering the solution $ \Phi $ of \eqref{eq:newcontrol} from $x$ to $ y, $ during $ [0, T ] .$ The associated cost functional is given by 
$$ V_T^{1, lin} (x, y ) = < e^{T A_1 } x - y , Q_T^{-1} ( e^{T A_1 } x - y ) > , $$
see Proposition 1.13 in Chapter 1 of \cite{coron}. A similar result applies to the second population, i.e.\ the system described by $ \Psi .$ 

We resume the above discussion and come back to the total process, consisting of the two populations. For any $x = (x_1 , \ldots , x_n ) $ a possible initial configuration of the two populations, and for any $ y \in \R^n ,$ we have therefore a control $ (u^1 (t) , u^2 (t)) ,$ such that the decoupled and linear system 
$ ( \Phi (t) , \Psi (t) ) $ solution of \eqref{eq:newcontrol} (and the analogous equation for the second population) is steered from $x$ to $y, $ during $ [0, T ].$ 

In what follows, we shall write $ \Phi (t) = ( \Phi_1 (t) , \ldots , \Phi_{n_1 +1} ( t) ) $ and $ \Psi ( t) = (\Psi_1 ( t) , \ldots , \Psi_{n_2 +1} (t) ) .$ 

In order to come back to the original control system, we put
\begin{equation}\label{eq:choiceh}
\dot h^1 (t) = \frac{u^1 (t) - c_1 f_2 (  \Psi_1 (t) ) }{c_1 / \sqrt{p_2} \; \sqrt{ f_2 (  \Psi_1 (t) ) } } \mbox{ and } \dot h^2 (t) = \frac{u^2 (t) - c_2 f_1 (  \Phi_1 (t) ) }{c_2 / \sqrt{p_1} \; \sqrt{ f_1 (  \Phi_1 (t) ) } }  .
\end{equation}
Since $f_1 $ and $f_2$ are lower bounded, $ \dot h^1 $ and $ \dot h^2 $ are well-defined and admissible, that is, $ \dot h^1 , \dot h^2 \in L^2_{loc}.$  Moreover, by the structure of \eqref{eq:newcontrol}, 
\begin{eqnarray*}
 \frac{ d \Phi_{ n_1 +1 } (t) }{dt} &= &- \nu_1 \Phi_{n_1 + 1 } ( t) + u^1 (t) \\
 &=& - \nu_1 \Phi_{n_1 + 1 } ( t)  + c_1 f_2 ( \Psi_1 (t) ) + \frac{c_1}{\sqrt{p_2}} \sqrt{ f_2 ( \Psi_1 (t) )} \dot h^1 (t) ,
\end{eqnarray*}
thus $ \varphi = ( \Phi, \Psi ) ,$ together with the choice of $ h $ in \eqref{eq:choiceh},  is solution of the original control problem \eqref{eq:generalcontrolsystem}. 
\end{proof} 
 
We use the ideas of the above proof to show that the cost functions $V (x,y) $ and $V_T (x,y) $ are upper semi continuous. 

\begin{theo}\label{theo:6}
Grant Assumption \ref{ass:1}. Then the  cost functions $ V_T ( x, y ) $ and $V(x, y) $ are upper semicontinuous in $x$ and in $y.$ 
\end{theo}

The main difficulty in the proof of this result is the fact that due to the hypo-ellipticity of the diffusion, the action of the drift is important in small time. As a consequence, if we want to steer the process within a small time step $\delta $ to any possible target point within a given ball, we have to take into account the action of the drift. It turns out that it is possible to steer the process from a fixed starting point $z$ to any point within a small ball around $ z + b(z) \delta ,$ and that the cost of doing this remains small, for small $ \delta .$ This is related to small time local controllability, see below, and also to the fact that the weak H\"ormander condition is satisfied. There is also a relation with density estimates of the associated diffusion over small time intervals, see e.g.\ Pigato \cite{Pigato}. In the proof we shall use tools developed in the recent paper by Delarue and Menozzi \cite{delaruemenozzi} where the same ``cascade''-structure of the drift as in our case is present.

\begin{proof} We fix some $\eta > 0 . $ Fix $T$ and $ x, y .$ Then there exists a control $h  $ such that $ \varphi^{ ( h , x) } ( T) = y $ and such that $I_{x, T } ( \varphi) \le V_T ( x, y ) + \eta .$ In the following we work with this fixed control and with the fixed trajectory $ \varphi := \varphi^{(h,x)}.$ 

Let us briefly explain the idea of our proof. We first show that for any $\delta > 0  $ and for any $\tilde x $ belonging to a small neighborhood of $x, $ it is possible to perturb the control $h$ on an interval $ [0, T - \delta ] $ such that the newly obtained controlled trajectory $\tilde \varphi $ stays within a small tube around $ \varphi  $ during $[0, T - \delta]  $ and such that the cost of doing so is comparable to the original cost $ I_{x, T - \delta } ( \varphi  ) .$ We then show that we can choose $\delta $ sufficiently small such that we are able to steer $\tilde \varphi $ from its position at time $T- \delta $ to any target position $ \tilde y $  belonging to a small neighborhood of $y, $ by maintaining the cost of doing so negligible. This last step will be done by relying on the ideas developed in the proof of the preceding theorem.  

{\bf Step 1.} We fixe some $0 < \delta < T $ and points $\tilde x , \tilde y $ in some neighborhoods of $x$ and $y.$ These neighborhoods and $ \delta $ will be chosen later. Write for short $ \gamma_1 (t) = \varphi_{n_1 +1} ( t) , $ $ \gamma_2 ( t) = \varphi_n ( t) $ for the two components of $ \varphi  $ depending directly on the control. In a first step of the proof, for a given $ \varepsilon , $ we choose any smooth trajectories $ \tilde \gamma_1 $ and $ \tilde \gamma_2 $ such that for all $0 \le t \le T - \delta , $
$$ \tilde \gamma_1 (t) \in B_\varepsilon ( \gamma_1 (t) ) , \tilde \gamma_2 (t) \in B_\varepsilon ( \gamma_2 (t) )  $$ 
and also 
$$  \frac{d}{dt} \tilde \gamma_1 (t) \in B_\varepsilon (\frac{d\gamma_1 (t) }{dt}  ) , \frac{d}{dt} \tilde \gamma_2 (t) \in B_\varepsilon ( \frac{d \gamma_2 (t)}{dt}  )  ,$$
with 
$$ \tilde \gamma_1 ( 0) = \tilde x_{n_1 +1 } , \;  \tilde \gamma_2 (0) = \tilde x_n .$$
We then put 
$$ \tilde \varphi_{n_1 + 1 } (t) := \tilde \gamma_1 (t) , \; \tilde \varphi_{n} (t) :=  \tilde \gamma_2 (t) ,$$
for all $ 0 \le t \le T - \delta .$ 
 
Once these two trajectories are fixed, by the structure of $b,$ we necessarily have
$$
 \tilde \varphi_{n_1  } (t)= e^{- \nu_1 t } \tilde x_{n_1} + e^{-\nu_1 t} \int_0^t e^{\nu_1 s} \tilde \gamma_1 ( s) ds .$$
Now, since $ \tilde x_{n_1} \in B_{\varepsilon } ( x_{n_1} ) $ and $ \tilde \gamma_1 ( s) \in  B_{\varepsilon } ( \gamma_1 (s)  ),$ for all $ s \le T - \delta , $ we certainly have that 
$$|  \tilde \varphi_{n_1  } (t) - 
(  e^{- \nu_1 t }  x_{n_1} + e^{-\nu_1 t} \int_0^t e^{\nu_1 s}  \gamma_1 ( s) ds  )| \le  \max( 1, \frac{1}{\nu_1 } ) \varepsilon   .$$ 
Thus, since $  e^{- \nu_1 t }x_{n_1} + e^{-\nu_1 t} \int_0^t e^{\nu_1 s}  \gamma_1 ( s) ds = \varphi_{n_1 } ( t) , $ 
$$ | \tilde \varphi_{n_1  } (t) -   \varphi_{n_1 } ( t) | \le  \max( 1, \frac{1}{\nu_1 } ) \varepsilon .$$ 

The same arguments apply for the other coordinates $ \tilde \varphi_i .$ 

As a consequence, introducing  $ \kappa  = \sqrt{n} \max( 1, \frac{1}{\nu_1^{n_1}}, \frac{1}{\nu_2^{n_2}  } ) ,$ where $ n = n_1 + n_2 + 2,$ we have constructed a trajectory $ \tilde \varphi (t)$ such that 
$$ \tilde \varphi (t)  \in B_{ \kappa \varepsilon } ( \varphi (t ) ) \mbox{ for all } t \le T- \delta .$$ 

The control which allows to produce this trajectory is given by 
$$\dot{ \tilde h}^1 (t) = \frac{\frac{d}{dt } \tilde \gamma_1 (t) + \nu_1  \tilde \gamma_1 ( t) - c_1 f_2 ( \tilde \varphi_{ n_1 + 2 } (t))  }{
c_1/ \sqrt{p_2} \; \sqrt{ f_2 ( \tilde \varphi_{ n_1 + 2} (t))} } , \;  \dot{\tilde h}^2 (t) = \frac{\frac{d}{dt } \tilde \gamma_2 (t) + \nu_2 \tilde \gamma_2 ( t) - c_2 f_1 ( \tilde \varphi_{ 1 } (t))  }{
c_2/ \sqrt{p_1} \; \sqrt{ f_1 ( \tilde \varphi_{  1 } (t))} } .$$
By continuity of $ f_1, f_2$ and the fact that $f_1,  f_2 $ are lower bounded, there exist $ \eta_1 = \eta_1 ( \varepsilon)  $ and $ \eta_2 = \eta_2 ( \varepsilon ) $ with $ \eta_1 ( \varepsilon ) \to 0, \eta_2 ( \varepsilon ) \to 0 $ as $ \varepsilon \to 0 , $  such that $ \dot{ \tilde h}^1 (t)  \in B_{\eta_1} ( \dot{  h}^1 (t) ) $ and $ \dot{ \tilde h}^2 (t) \in B_{\eta_2} ( \dot{  h}^2 (t) ), $ for all $t \le T - \delta .$ We choose $ \varepsilon_1 $ such that 
\begin{equation}\label{eq:choiceofepsilon}
 (T- \delta ) [(\eta_1 ( \varepsilon))^2 + (\eta_2 ( \varepsilon))^2] \le \eta \mbox{ for all }  \varepsilon \le \varepsilon_1.
\end{equation}
Then clearly 
$$ I_{\tilde x, T- \delta } ( \tilde \varphi ) \le I_{x, T - \delta } ( \varphi )   + \eta .$$ 
For the moment we have produced a controlled trajectory $ \tilde \varphi $ steering the initial point $ \tilde x $ belonging to  $B_\varepsilon (x)$ to a point $ \tilde z = \tilde \varphi (T- \delta ) \in B_{\kappa \varepsilon} ( z ) , $ where $z = \varphi (T- \delta ) ,$ such that we have a control on the cost function of this new trajectory. \\

{\bf Step 2.} Consider now the original control system $ \varphi = \varphi^{(h,x) } $ on the interval $ [T - \delta , T].$ Its coordinate $ \varphi_{n_1 +1} $ solves the equation 
$$ \dot \varphi_{n_1+1 } ( t) = - \nu_1 \varphi_{n_1+1 } ( t) + c_1 f_2 ( \varphi_{n_1 + 2 } (t)) + \frac{c_1}{ \sqrt{p_2}} \; \sqrt{ f_2 (  \varphi_{ n_1 + 2 } (t))}  \dot h_t^1  .$$ 
If we write 
\begin{equation}\label{eq:decouplefirst}
 u^1 (t) := c_1 f_2 ( \varphi_{n_1 + 2 } (t)) + \frac{c_1}{ \sqrt{p_2}} \; \sqrt{ f_2 (  \varphi_{ n_1 + 2 } (t ))}  \dot h_{t } ^1  ,
\end{equation}
then clearly, $ u^1 \in L^2 ( [0, T ] ) $ and 
\begin{equation}\label{eq:decouplesecond}
 \dot \varphi_{n_1+1 } ( t) = - \nu_1 \varphi_{n_1+1} ( t)  + u^1 (t)  , T- \delta  \le t \le T , 
\end{equation}
with $\varphi_{n_1+1} ( T - \delta) = z, \varphi_{n_1+1 } (T) = y. $ The same argument applies for $ \dot \varphi_n ( t) , $ with the definition $ u^2 (t) = c_2 f_1 ( \varphi_{ 1 } (t)) + (c_2/ \sqrt{p_1}) \; \sqrt{ f_1 (  \varphi_{ 1 } (t ))}  \dot h_{t } ^2 .$

Since $\dot h^1 , \dot h^2 \in L^2 ( [0, T ] )  $ and since $ f_1$ and $ f_2 $ are bounded,  we can now choose $ \delta $ such that 
\begin{equation}\label{eq:choicedelta}
\frac12 \int_{T- \delta }^T (u^1 (t) )^2 + (u^2 (t) )^2 dt \le \underbar f \eta , \;  \delta (\| f_1 \|_\infty  + \| f_2 \|_\infty ) \le \eta  ,
\end{equation}
where we recall that $ \underbar f$ is such that $ f_1 ( x) \geq \underbar f , $ $ f_2 (x) \geq \underbar f , $ for all $ x \in \R .$ 

With this choice of $ u^1 $ and $u^2 $ we can rewrite the control problem on $ [T- \delta , T ] $ as in the proof of Theorem \ref{theo:control}. As there, we put $ \varphi (t) = ( \Phi (t) , \Psi (t) ) $ with $ \Phi (t ) = (\varphi_1(t) , \ldots , \varphi_{n_1 +1 } (t) ) $ and $ \Psi (t) = (\varphi_{n_1+2 } (t), \ldots, \varphi_n (t) ) .$ Then 
$$
\dot \Phi (t) = F_1 ( \Phi ( t) ) + B_1  u^1 ( t) ,\; \dot \Psi (t) = F_2 (\Psi (t) ) + B_2 u^2 (t) ,$$
where $B_1, B_2, F_1 , F_2$ are defined in \eqref{eq:B} and \eqref{eq:F} above. 

In what follows, by abuse of notation, we will systematically write $ z= (z_1, \ldots , z_{n_1 +1} ) $ for starting configurations of $ \Phi ( t) $ or $ z = ( z_{n_1 + 2} , \ldots , z_n ) $ for those of $\Psi (t) $ or $ z = (z_1 , \ldots , z_n ) $ for starting configurations of the entire system, depending on the context.

Having thus constructed a specific controlled trajectory, we certainly have that 
\begin{equation}\label{eq:good}
 \frac12 \int_{T- \delta}^T (u^1 (t) )^2 dt \geq  V^{1,lin}_{\delta} (z,y) , \; \frac12 \int_{T- \delta}^T (u^2 (t) )^2 dt \geq  V^{2,lin}_{\delta} (z,y) , 
\end{equation} 
where $ V^{1, lin}_{ \delta } (z,y) =  \inf \{ \frac12 \int_{T-\delta}^T (u^1 (t) )^2 dt  : \Phi (T - \delta ) = z, \Phi ( T) = y \} $ such that $ \dot \Phi (t) = F_1 ( \Phi ( t) ) + B_1 u^1  (t) ,$ and with a similar definition for the second system.  

{\bf Step 3.} The key observation is now that system \eqref{eq:newcontrol} satisfies the conditions of Section 4.1 of Delarue and Menozzi  \cite{delaruemenozzi} (with order of coordinates reversed, i.e.\ the coordinate depending on the noise is the first in \cite{delaruemenozzi} and not the last as it is the case here). The main point is the {\it cascade-}structure of the drift, i.e., the fact that the $i-$th coordinate of $ F_1 (x), $ which is given by $   - \nu_1 x_i + x_{i+1} ,$ does only depend on the coordinates $ x_i $ and $x_{i+1} ,$ for all $ 1 \le i \le n_1 +1 .$  In particular, writing $T_t $ for the $(n_1 +1) \times (n_1+1) -$diagonal matrix having entries 
$$ T_t = diag ( t^{ n_1+1 }, t^{n_1 }, \ldots , t) ,$$
Proposition 4.1 of \cite{delaruemenozzi} implies that there exists a constant $C_1$ depending only on $T,$ such that 
\begin{equation}\label{eq:explode?}
   V^{1, lin}_{\delta} (z,y) \geq C_1 \delta | T_\delta^{-1} ( \theta_\delta ( z) - y) |^2 , 
\end{equation}
where $ \theta_\delta  $ is the deterministic flow associated to the zero-noise system $ \dot \theta_ t (z)  = F_1 ( \theta_t (z) ) , $ and where $ z  = (z_1 , \ldots, z_{n_1+1 } ), $ $ y  = (y_1 , \ldots , y_{n_1+1 } ) .$ 



The important point is now that as a consequence of \eqref{eq:good} together with \eqref{eq:choicedelta}, we have 
\begin{equation}\label{eq:first}
C_1 \delta | T_\delta^{-1} ( \theta_\delta ( z) - y) |^2  \le \underbar f \; \eta .
\end{equation} 
The same argument applies to the second population.

{\bf Step 5.} Recall the definition of $\varepsilon_1 $ in \eqref{eq:choiceofepsilon}. We now choose $ \varepsilon_2 \le \varepsilon_1 $ such that for all $\varepsilon \le \varepsilon_2, $ for all $ \tilde z \in B_{ \kappa \varepsilon }( z) , $ $ \tilde y \in B_\varepsilon ( y) , $  
\begin{equation}\label{eq:second}
 \delta | T_\delta^{-1} ( \theta_\delta ( \tilde z) - \tilde y) |^2  \le \frac{2}{C_1} \underbar f \;  \eta .
\end{equation} 
We then solve \eqref{eq:newcontrol} on $[T- \delta , T ]  $ and obtain a system $ \tilde \Phi ( t) $  with $\tilde \Phi ( T- \delta ) = \tilde z $ and $ \tilde \Phi ( \delta ) = \tilde y ,$ for any $\tilde  y \in B_\varepsilon ( y ) .$ By Proposition 4.2 of \cite{delaruemenozzi}, this is possible using a control $\tilde  u^1 $ such that 
$$ \sup \{  ( \tilde u^1  (s))^2 , T- \delta \le s \le T \}  \le C_2  | T_\delta^{-1} ( \theta_\delta ( \tilde z) - \tilde y) |^2 
\le \frac{2C_2 }{C_1 \delta } \underbar f \;  \eta,$$
where $C_2 $ is another universal constant and where we have used \eqref{eq:second}. In particular, 
$$ \frac12 \int_{T- \delta }^T  ( \tilde u^1  (s))^2 ds \le \frac{C_2 }{C_1}  \underbar f \; \eta .$$

The same argument applies to $\Psi (t),$ describing the second population of particles. In order to come back to the original control system, we use \eqref{eq:decouplefirst} and find 
$$\dot {\tilde h}_t^1 = \frac{\tilde u^1 (t) - c_1 f_2 ( \tilde \Psi_1 (t) ) }{[c_1 / \sqrt{p_2}] \; \sqrt{ f_2 ( \tilde \Psi_1 (t) ) } } . $$ 
Then 
$$ \frac12 \int_{T- \delta}^T ( \dot {\tilde h}_t^1)^2 dt \le \frac{p_2}{c_1^2} \frac{1}{\underbar f} \frac{C_2 }{C_1} \underbar f  \eta + p_2 \| f_2\|_\infty \delta \le C \eta , $$
for some constant $C$ not depending on $f,$ by the choice of $\delta $ in \eqref{eq:choicedelta}. 

Summarizing the above arguments, we have thus constructed a control $ (\dot{ \tilde h}^1 , \dot{\tilde h}^2 )$ acting on $ [T - \delta , T ]$ steering $ \tilde z $ to $ \tilde y $ for any $ \tilde y \in B_\varepsilon ( y) , $ at a cost at most $ C \eta .$ Pasting together the two control paths $\tilde \varphi $ constructed in Step 1 on $ [0, T - \delta ] $ and the last one, we have thus obtained a path $ \tilde \varphi $ from $ \tilde x $ to $\tilde y$ at a total cost 
$$ I_{\tilde x, T} ( \tilde \varphi ) \le I_{x, T- \delta } ( \varphi ) + \eta + C \eta \le I_{x, T } ( \varphi )  + (1 + C  ) \eta . $$
Since $ \eta $ can be chosen arbitrarily small, this implies that $ V_T (x, y ) $ is upper semicontinuous in $x$ and in $y.$ The fact that $ V(x, y ) $ is upper semicontinuous in $x $ and $y$ follows then easily from this. 
\end{proof}

{\bf Small time local controllability.}
We will now discuss the important notion of small time local controllability which is related to the behavior of the system close to equilibrium points or to periodic orbits.

In the following, we restrict attention to controls $h $ such there is some -- sufficiently fine -- finite partition $0=s_0 < s_1 < \ldots < s_\nu = t$ such that all components  $\dot{ h}^\ell$  are smooth on $s_{r-1}$ and $s_r$. We shall call such controls ${ h}$  {\em piecewise smooth.} 
We denote $R_{  \tau} ( x) $ the set of points which can be reached from $x$ in time $ \tau $ using a piecewise smooth control $ h;$ i.e.\ 
$$R_{\tau}(x)  = \{ \varphi^{( { h}, x)}(\tau)   : \, h\in{\tt H}\;\;\mbox{piecewise smooth}\;  \}.$$

We shall also consider 
$$ R^{ M}_{ \tau}(x)  = \{ \varphi^{({ h}, x)}(\tau)   : \,{ h}\in{\tt H}\;\;\mbox{piecewise smooth}, \| \dot { h} \|_\infty =  \| \dot h^1 \|_\infty + \| \dot h ^2 \|_\infty \le M \;  \}.$$

We say that the system is {\it small-time locally controllable at $x$} if $ R_{  \tau} ( x) $ contains a neighborhood of $x$ for every $ \tau > 0.$ 

\begin{lem}\label{cor:STLCxstar}
$Y^N$ is small-time locally controllable at $x^*.$ 
\end{lem}

\begin{proof}
Write $ \sigma^1 $ and $ \sigma^2$ for the two columns of the diffusion matrix $ \sigma .$ Then it is straightforward to verify that $ \sigma^1, [\sigma^1 , b ] , [[\sigma^1, b ] , b ] , \ldots $ and $\sigma^2, [\sigma^2 , b ] , [[\sigma^2, b ] , b ] \ldots $ span $ \R^n .$ In particular, the system satisfies the weak H\"ormander condition. Then the assertion follows from Theorem 3.4 of Lewis \cite{Lewis}, based on the results of Sussmann \cite{Sussmann} and Bianchini and Stefani \cite{Bianchini}.
\end{proof}

The following theorem states a result concerning the small time controllability around points which are on a periodic orbit $\Gamma $ of the limit system \eqref{eq:cascade}. On $\Gamma , $ the drift vector $ b  $ plays an important role, in the sense of a ``shift'' along the orbit. As a consequence, the system is not small time locally controllable in the classical sense, but in a ``shifted sense'' as stated in the following theorem.

\begin{theo}\label{theo:stlc}
Grant Assumptions \ref{ass:1} and \ref{ass:4}. Let $\Gamma $ be a periodic orbit of \eqref{eq:cascade}, $x_0 \in \Gamma ,$
and let $x^{x_0}(t) $ be the solution of \eqref{eq:limit}, issued from $ x_0  $ at time $0.$ Then there exists $ \delta^* $ such that for any $0< \delta < \delta^* ,$ for all $M ,$ we have that $x^{x_0} (\delta)   $ is in the interior of $R^{M}_\delta (x_0) .$
\end{theo}

The proof of this theorem is given in the Appendix. 

With these results at hand we are able to prove the following proposition which is the analogue of Proposition 3 of Rey-Bellet and Thomas \cite{reybellet}. \cite{reybellet} consider systems locally around equilibria, and therefore, the drift vector does not play a role in their case. In our case, we have to consider the control system locally around non constant periodic orbits -- hence the drift vector does play a crucial role since it induces a shift along the orbit which is not negli-geable in the study of the system. We shall use the following notation. For any periodic orbit $ \Gamma $ of \eqref{eq:cascade}, let $B_\varepsilon ( \Gamma ) = \{ x \in \R^n : dist (x, \Gamma ) < \varepsilon \} .$ 

\begin{prop}\label{prop:STLCgamma}
Grant Assumptions \ref{ass:1} and \ref{ass:4}. Let $\Gamma $ be a periodic orbit of \eqref{eq:cascade}. For any $\eta > 0 $ and $\varepsilon' > 0 $ there exists $ \varepsilon  > 0 $ such that $ \varepsilon < \varepsilon' /3 $ with the following properties. For all $x , y \in B_\varepsilon ( \Gamma )  , $ there exist $T> 0 $ and a control ${ h } \in {\tt H} $ such that $ \varphi^{( h , x)} ( T) = y ,$ $ \varphi^{ ( h , x)} ( s) \in B_{ 2 \varepsilon ' /3}( \Gamma ) $ for all $ 0 \le s \le T,$ and $ I_{x, T} ( \varphi^{ ( \tt h , x)}  ) \le \eta .$ 
\end{prop}

\begin{proof} 
We start by introducing some additional objects needed in the proof. We denote by $ \tilde \varphi^{(h, x) } $ the inverse flow,  solution of 
$$ d \tilde \varphi^{(h, x )} (t) = - b ( \tilde \varphi^{(h, x )} (t) ) dt + \sigma ( \tilde \varphi^{(h, x )}  (t) ) \dot h (t) dt  , \tilde \varphi^{(h, x )} (0 ) =x ,$$
and write $ \tilde R^M_\delta (x)  $ for the set of attainable points for the inverse flow, using a piecewise smooth control which is bounded by $M.$

We then choose $M$ and $ \delta_0  $ such that $M^2 \delta_0 < \eta $ and such that $ R_{\delta}^{M} ( x) \subset B_{ 2\varepsilon ' /3} ( \Gamma ) $ and $ \tilde R_{\delta}^{M} ( x) \subset B_{ 2\varepsilon ' /3} ( \Gamma  ) $ for all $ \delta \le \delta_0, $ for all $ x \in  \Gamma .$  In the sequel, $ \delta \le \delta_0 $ will be fixed. 

For any $ z \in \Gamma , $ there exist $ \varepsilon_1 (z)  $ such that $B_{\varepsilon_1 (z) } ( x^z (\delta) ) \subset R_{\delta}^{ M} ( z)  ,$ by Theorem \ref{theo:stlc}. 

Notice that $ x^z (\delta) \in \Gamma $ if $ z \in \Gamma .$ Notice moreover that for any $ v \in \Gamma $ there exists $ w = x^v (- \delta ) \in \Gamma $ such that $ x^w (\delta) = v .$ Therefore, by compactness of $\Gamma, $ there exists a finite collection $z_1, \ldots , z_K \in \Gamma $ such that $ \Gamma \subset \bigcup_{k=1}^K B_{\frac13 \varepsilon_1 ( z_k) } ( x^{z_k} (\delta)  )$ and such that for all $k, $ $  B_{ \varepsilon_1 ( z_k) } ( x^{z_k} (\delta)  ) \subset R_\delta^M  (z_k) .$ 

Applying the same arguments as above to the inverse flow,  for all $x \in \Gamma $ there exists $ \varepsilon_2 ( x)$ such that $B_{\varepsilon_2 (x)  } ( x^x (-\delta) ) \subset \tilde R_{\delta}^{ M} ( x)  .$ 
Then again, by compactness of $\Gamma,$  $ \Gamma \subset \bigcup_{k=1}^K B_{\frac13 \varepsilon_2 ( u_k) } ( x^{u_k} (-\delta)  )$ for $ u_1 , \ldots , u_K \in \Gamma $ (where we suppose w.l.o.g.\ that the number of balls is the same in the two coverings) and such that $B_{ \varepsilon_2 ( u_k) } ( x^{u_k} (-\delta)  )
\subset \tilde R^M_\delta (u_k) $ for all $k.$
 
Choose now 
$$ \varepsilon \le   \min \{ \varepsilon_1 ( z_k )/4  , k \le K \}  \wedge \min \{ \varepsilon_2 ( u_k )/4  , k \le K \} \wedge \varepsilon'/3 .$$  
Let $y \in B_ \varepsilon ( \Gamma ).$ Then there exists $ y^* 
\in \Gamma $ such that $  \| y  - y^* \| \le \varepsilon .$ Let $k$ be such that $ \| y^* - x^{z_k } (\delta)  \| \le  \varepsilon_1 ( z_k )/3 .$ Hence, $\| y - x^{z_k } (\delta) \| <  \varepsilon_1 ( z_k )  $ and as a consequence, $ y \in R_{\delta}^{ M} ( z_k) .$  In the same way, for any $ x \in B_\varepsilon ( \Gamma ) $  there exists $ u_l $ such that $ x \in \tilde R_\delta^{ M} ( u_l ) .$  Therefore, there exist $ { h }_1  $ and ${ h}_2 $ with $ \| \dot h_1 \|_\infty \le M,  \| \dot h_2 \|_\infty \le M ,$ such that 
$  \varphi^{ ( h_2 , z_k) } (\delta) = y  $ and $ \tilde \varphi^{ ( h_1 , u_l ) } (\delta) = x .$ 

By reversing the time, this yields a trajectory $ \varphi^{( h_1 , x) }  (t) $ with $ \varphi^{( h_1 , x)}  (0)  = x $ and  $\varphi^{(  h_1 , x)}  (\delta)  = u_l . $ Then it suffices to choose $ T$ such that $ x^{ u_l}_{T - 2 \delta } = z_k $ -- this is just a shift on the orbit. 

To finish the proof, observe that by construction the produced trajectory $ \varphi = \varphi^{( h_1 , x) }   $ is such that $\varphi ( s)  \in  B_{ 2 \varepsilon ' /3 } ( \Gamma ) $ for all $ s \le T ,$ since we have chosen $M$ and $ \delta $ such that $ R_{\delta}^{M} ( x) \subset B_{ 2\varepsilon ' /3} ( \Gamma ) $ and $ \tilde R_{\delta}^{M} ( x) \subset B_{ 2\varepsilon ' /3} (\Gamma  ) $ for all $ \delta \le \delta_0, $ for all $ x \in \Gamma .$
\end{proof}

\section{Large deviations and asymptotics of the invariant measure}\label{sec:6} 
Recall that we have introduced controlled trajectories
$$
\varphi = \varphi^{({h}, x)} \; \mbox{solution to}\; d \varphi (t) = b (  \varphi (t) ) dt +  \sigma( \varphi (t) ) \dot h(t) dt, \; \mbox{with $\varphi (0)=x,$}  
$$
together with their rate function  on time intervals $[0, t_1]$ 
$$
I_{x, t_1} (f) = \inf_{ {h} \in {\tt H} : \varphi^{( { h}, x)}  (t) = f(t) , \; \forall t \le t_1 } \frac12 \int_0^{t_1} [| \dot{ h}^1(s)|^2+ | \dot{ h}^2 (s) |^2] ds  .
$$
This rate function is not explicit since the diffusion matrix $\sigma $ is degenerate. It is however a ``good rate function", i.e.\ all of its level sets $ \{ f : I_{x, t_1 } ( f) \le \alpha \} $ are compact, and the following large deviation principle for the sample paths of the diffusion $ Y^N$ is well known, going back to Freidlin and Wentzell \cite{FW}. We quote if from \cite{DZ}. 

\begin{theo}[Corollary 5.6.15 of \cite{DZ}]
Grant Assumption \ref{ass:1}. Let $Y^N $ denote the solution of \eqref{eq:diffusionsmallnoise}, starting from $ x \in \R^n .$ Then for any $x \in \R^n $ and for any $ t_1 < \infty , $ the rate function $ I_{x, t_1 } ( f) $ is a lower semicontinuous function on $ C ( [0, t_1 ], \R^n ) $ with compact level sets. Moreover, the family of measures $ Q_x^N $ satisfies the large deviation principle on $ C ( [0, t_1 ], \R^n ) $ with rate function $ I_{x, t_1} ( f) .$ \\
(i ) For any compact $ K \subset \R^n $ and any closed $ F \subset C ( [0, t_1 ], \R^n ) , $ 
\begin{equation}\label{eq:compact}
\limsup_{ N \to \infty } \frac1N \log \sup_{x \in K } Q_x^N ( F) \le - \inf_{x \in K} \inf_{ f \in F} I_{x, t_1} ( f) .
\end{equation}
(ii) For any compact $ K \subset \R^n $ and any open $ O  \subset C ( [0, t_1 ], \R^n) , $ 
\begin{equation}\label{eq:open}
\liminf_{ N \to \infty } \frac1N \log \inf_{x \in K } Q_x^N ( O) \geq - \sup_{x \in K} \inf_{ f \in O} I_{x, t_1} ( f) .
\end{equation}
\end{theo}
 
\subsection{Proof of Theorem \ref{theo:main}}
We are now able to give the proof of our main result, Theorem \ref{theo:main}. It follows closely Freidlin and Wentzell \cite{FW}, adapted to the situation of degenerate diffusions in Rey-Bellet and Thomas \cite{reybellet}. 

Recall that $ K = \{ x^* \} \cup \bigcup_{l=2}^L K_l $ denotes the $\omega-$limit set of \eqref{eq:generalcontrolsystem}. To start, we stress that the diffusion process $Y^N $ solution of \eqref{eq:cascadeapprox} satisfies the two main assumptions of \cite{reybellet} which are the following.

\begin{ass}
The diffusion process $Y^N ( t) $ has a hypo-elliptic generator, and for any $x$ belonging to the $\omega-$limit set $K ,$ the control system associated with 
\eqref{eq:generalcontrolsystem} is small-time locally controllable (in a sense of a shift along periodic orbits, as stated in Theorem \ref{theo:stlc}).
\end{ass}

\begin{ass}
The diffusion process is strongly completely controllable and for any $ T > 0, $ the cost function $V_T ( x,y) $ is upper semicontinuous in $x$ and $y.$ 
\end{ass}

We now follow Freidlin-Wentzell \cite{FW} and put 
$$ U = B_\varepsilon ( K) , V = B_{\bar \varepsilon } ( K), $$
for $\varepsilon < \bar \varepsilon $ such that $ 3 \varepsilon < \bar \varepsilon .$ We introduce 
$$ \tau_0 = 0, \sigma_n = \inf \{ t > \tau_n : Y^N (t) \in  V^c\} , \tau_{n+1}  = \inf \{ t > \sigma_{n} : Y^N ( t) \in  U \} , n \geq 0 .$$
Since $Y^N$ is Harris-recurrent with invariant measure $\mu^N $ being of full support and therefore charging $ U, V$ and $V^c, $ we have $ \tau_n < \tau_{n +1} < \infty , \sigma_n < \sigma_{n+1} < \infty  $ almost surely,  and $ \sigma_n , \tau_n \uparrow \infty $ as $ n \to \infty . $ Writing $ U_n := Y^N ( \tau_n ) , n \geq 1 ,$ $U_n $ is a Markov chain taking values in $ \partial U $ which is a compact set. In particular, $( U_n)_n$ admits a (unique) invariant probability measure $\ell_N $ on $ \partial U $ (since $Y^N$ is Harris), and the invariant measure $ \mu^N $ of the process $ Y^N$ can be decomposed as 
$$ \mu^N ( D) =\frac{1}{ c (N)} \int_{\partial U} \ell_N (dx) E^N_x \int_0^{\tau_1} 1_D ( Y^N ( t) ) dt =: \frac{1}{ c (N)}  \nu^N ( D) ,$$
where 
$$c(N) = \int_{\partial U} \ell_N (dx) E^N_x \tau_1 .$$ 

We now take a regular open set $ D,$ i.e.\ a set such that $ \partial D$ is a piecewise smooth manifold,  with $ dist (D, K ) > \Delta .$  Let $ \tau_D = \inf \{ t> 0  : Y^N ( t) \in D \} $ be the associated hitting time. Then we have the following result.

\begin{lem}\label{lem:1}
Grant Assumptions \ref{ass:1} and \ref{ass:4}. Let 
$$ S := \inf \{ t \geq 0 : Y^N \in B_{\varepsilon } ( K) \cup D \} .$$
Then for any compact set $E, $ 
$$ \lim_{T \to \infty } \limsup_{ N \to \infty }  \sup_{x \in E} Q^N_x ( S > T ) = 0 .$$
\end{lem}

\begin{proof}
We have 
$$ Q^N_x ( S > T )  \le \frac1T E^N_x \tau_\varepsilon .$$
But by Proposition \ref{prop:lyapunov}, $ \sup_N  E^N_x  \tau_\varepsilon  \le C G ( x) ,$ where $G$ does not depend on $N.$ The fact that $G$ is bounded on the compact set $E$ then implies the result.  
\end{proof}

In the following we establish two classical results on the growth rate of the expected escape time $ E^N_x \sigma_0 $  that will be useful in the sequel. They are analogous to the results of \cite{FW}, transposed to the hypo-elliptic context of our model.

\begin{prop}\label{prop:sigma_0first}
Grant Assumptions \ref{ass:1} and \ref{ass:4}.  Given $h > 0 ,$ for $ \varepsilon < \bar \varepsilon $ such that $ 3 \varepsilon < \bar \varepsilon $ sufficiently small, 
$$\liminf_{N \to \infty} \frac1N \log  \inf_{x \in \partial B_\varepsilon (K) } E^N_x \sigma_0 \geq  - h  .$$ 
\end{prop}

An analogous result holds for the upper bound of $ E^N_x \sigma_0 .$ 

\begin{prop}\label{prop:sigma_0second}
Grant Assumptions \ref{ass:1} and \ref{ass:4}. 
Given $h > 0 ,$ for $ \varepsilon < \bar \varepsilon $ such that $ 3 \varepsilon < \bar \varepsilon $ sufficiently small, 
$$\limsup_{N \to \infty} \frac1N \log  \sup_{x \in \partial B_{\varepsilon} (K ) } E^N_x \sigma_0 \le   h  .$$ 
\end{prop}

The proofs of the two propositions are given in the Appendix. 

Recall that the $\omega -$limit set $K= \{ x^*\}  \cup \bigcup_{l=2}^L K_l $ is divided into disjoint subsets consisting of equivalence classes induced by the equivalence relation $ x \sim y ,$ where we say that $ x \sim y $  if and only if $ V (x, y) = V ( y, x ) = 0.$ Following \cite{FW}, we now introduce 

$$ \tilde V(K_i , K_j ) =  \inf_T \inf \{ I_{x, T } ( \varphi ) : \varphi ( 0) \in K_i, \varphi (T) \in K_j, \varphi ( t) \notin \bigcup_{ l \neq i, j } K_l , 0 \le t\le T \},$$
$$ \tilde V (K_i, z) = \inf_T \inf \{ I_{x, T } ( \varphi ) : \varphi ( 0) \in K_i, \varphi (T) = z, \varphi ( t) \notin\bigcup_{ l \neq i } K_l , 0 \le t\le T \} ,$$
for $i = 1, 2, \ldots , L.$ 
We also put 

$$ \tilde V_i (x, y ) = \inf_T \inf \{ I_{x, T } ( \varphi ) : \varphi ( 0) = x , \varphi (T) = y, \varphi ( t) \notin \bigcup_{ l \neq i } K_l ,,  0 \le t\le T \} .$$

As a consequence of the small-time local controllability as stated in Corollary \ref{cor:STLCxstar} and of Proposition \ref{prop:STLCgamma} we have the following useful result.

\begin{lem}\label{lem:STLC}
Grant Assumptions \ref{ass:1} and \ref{ass:4}.  For all $ i, $ for all $ x, y \in K_i ,$ and for all $h$ there exists $ \delta $ such that $ | x - \tilde x | < \delta , $ $|y - \tilde y | < \delta $ imply that $ \tilde V_i ( \tilde x, \tilde y ) < h .$ 
\end{lem}

We put $ B_\varepsilon (K_i ) = \{ y \in \R^n : dist (y, K_i ) < \varepsilon \} , $ for $ i =1, 2, \ldots, L.$ We quote the following lemma from \cite{reybellet}.

\begin{lem}[Lemma 4 of \cite{reybellet}]\label{lem:4}
Grant Assumptions \ref{ass:1} and \ref{ass:4}. 
For any $ h > 0 $ there exist $ \varepsilon < \bar \varepsilon $ sufficiently small such that 
$$ \limsup_{N \to \infty } \frac1N \log \sup_{ x \in \partial B_{\bar \varepsilon }(K_i) } Q_x^N ( \tau_D < \tau_1) \le 
- \left( \inf_{ z \in D} \tilde V ( K_i , z ) - h \right) $$
and  
$$  \limsup_{N \to \infty } \frac1N  \log \sup_{ x \in \partial B_{\bar \varepsilon }(K_i) } Q_x^N ( Y^N ( \tau_1 ) \in \partial B_{ \varepsilon }(K_j) ) \le - 
\left( \tilde V( K_i, K_j ) - h \right) .$$
\end{lem}

\begin{proof}
Once Lemma \ref{lem:STLC} established, the proof is the same as the proof of Lemma 4 of \cite{reybellet}. The fact that most of the sets $ K_i$ are periodic orbits does not change the proof. 
\end{proof}

Small time local controllability around $x^* $ and around periodic orbits $\Gamma $ are also sufficient to obtain the lower bound obtained by \cite{reybellet} in their Lemma 5: 

\begin{lem}[Lemma 5 of \cite{reybellet}]
Grant Assumptions \ref{ass:1} and \ref{ass:4}. 
For any $ h > 0 ,$ for any $ \varepsilon < \bar \varepsilon $ sufficiently small,
$$ \liminf_{N \to \infty } \frac1N  \log \inf_{ x \in \partial B_{\bar \varepsilon }(K_i)} Q_x^N ( \tau_D < \tau_1) \geq 
- \left( \inf_{ z \in D} \tilde V ( K_i , z ) + h \right) $$
and 
$$  \liminf_{N \to \infty }\frac1N  \log \inf_{ x \in \partial B_{\bar \varepsilon }(K_i) } Q_x^N ( Y^N ( \tau_1 ) \in \partial B_{ \varepsilon }(K_j) ) \geq - 
\left( \tilde V( K_i, K_j ) + h \right) .$$
\end{lem}

Also, the lower bound of Lemma 6 of \cite{reybellet} is easily verifiable in our context, and we obtain

\begin{lem}[Lemma 6 of \cite{reybellet}]\label{lem:6}
Grant Assumptions \ref{ass:1} and \ref{ass:4}. 
For any $ h > 0 , $ $ \liminf_{N \to \infty } \frac1N \log \nu^N ( \R^n ) \geq - h .$ 
\end{lem}

\begin{proof}
The proof is the same as in \cite{reybellet}, once we have obtained the estimate 
\begin{equation}
\inf_{ x \in \partial B_\varepsilon (K) } E^N_x ( \sigma_0 ) \geq e^{ - N h} ,
\end{equation}
as proven in Proposition \ref{prop:sigma_0first}. 
\end{proof}

In order to finish the proof of our main theorem, we follow now closely Rey-Bellet et Thomas \cite{reybellet} and Freidlin and Wentzell \cite{FW}.  

1) We have, as in formula (46) of \cite{reybellet}, 
\begin{multline*}
\nu^N ( D) \le \sum_{i=1}^L \ell^N ( \partial B_\varepsilon (K_i ) ) \sup_{x \in \partial B_\varepsilon (K_i ) } E^N_x \int_0^{\tau_1} 1_D ( Y^N_s) ds \\
\le L \max_i \ell^N ( \partial B_\varepsilon (K_i) ) \sup_{x \in \partial B_\varepsilon (K_i) }  Q^N_x ( \tau_D \le \tau_1 ) \sup_{y \in \partial D} E^N_y \tau_1 .
\end{multline*}
But $\sup_{y \in \partial D} E^N_y \tau_1 \le C ,$ for some fixed constant $C,$ by Proposition \ref{prop:lyapunov}. Moreover, we have, for sufficiently small $ \varepsilon < \bar \varepsilon, $ by Lemma \ref{lem:4},  for $x \in  \partial B_\varepsilon (K_i ) ,$ 
$$  Q^N_x ( \tau_D \le \tau_1 ) \le \exp ( - N  ( \inf_{z \in D} \tilde V (K_i, z ) - h/4 ) ) .$$
Define now the function  $ \tilde W ( x) $ in the same way as $ W(x) $ in \eqref{eq:W}, by replacing all $ V (K_m, K_n ) $ by $ \tilde V( K_m, K_n ) .$ By Freidlin-Wentzell \cite{FW}, Lemma 3.1 and 3.2 together with Lemma 4.1 and 4.2 of Chapter 6, we know that $\tilde W(x) = W(x),$  and therefore we obtain 
$$ \ell^N ( \partial B_\varepsilon (K_i) ) \le \exp ( - N [ W ( K_i ) - \min_j W (K_j)  - h/4 ] ) ,$$
for sufficiently large $N.$  As a consequence, following the lines of proof of  \cite{reybellet}, (47)--(50),
$$ \nu^N ( D) \le L C \exp ( - N [ \inf_{z \in D} W(z) - h/2 ] ) .$$ 
Finally, using the lower bound obtained for $ \nu^N ( \R^n ) $ in Lemma \ref{lem:6}, implying that 
$$ \nu^N ( \R^n) \geq \exp ( - \frac{Nh}{2} ) $$
for all $ N $ sufficiently large, we obtain 
$$ \mu^N ( D) \le LC \exp ( - N [ \inf_{z \in D} W(z) - h ] ),$$
concluding the first part of the proof. 

We now turn to the study of the lower bound in \eqref{eq:main}. We fix some $ \delta > 0 $ sufficiently small such that $ D_\delta = \{ x \in D : dist (x, \partial D) \geq \delta \} $ satisfies $ D_\delta \neq \emptyset .$ Let $ z \in D $ and fix $ i$ such that $ \tilde V ( K_i, z ) < \infty .$ Such an index $i$ always exists due to  the complete controllability property.\footnote{Indeed, for any $ j , $ $V (K_j, z ) < \infty .$ Suppose that the trajectory achieving the minimal cost to go from $K_j $ to $z $ visits the sets $K_j, $ followed by $ K_{n_1}, \ldots , K_{n_l} , $ before leaving the last of them, $K_{n_l}  ,$ and reaching the target $z.$ It is then sufficient to choose $i$ to be equal to the index of the last visited set, that is, $ i := n_l .$ } The proof of Theorem \ref{theo:6} shows that it is possible to choose $ \delta  $ so small that 
$$ \inf_{ z \in D_\delta } \tilde V ( K_i, z ) \le \inf_{z \in D} \tilde V ( K_i, z) + h/ 4.$$
This point is crucial for the rest of the proof.

Then
$$ \nu^N ( D) \geq \min_i \left[ \ell^N ( \partial B_\varepsilon(K_i ) ) \inf_{x \in \partial B_\varepsilon (K_i ) } Q^N_x ( \tau_{D_\delta} < \tau_1 ) \right]  \inf_{x \in \partial D_\delta } E^N_x \int_0^{\tau_1} 1_D ( Y^N_s) ds .$$

We will prove below that 
\begin{equation}\label{eq:last}
\inf_N  \inf_{x \in \partial D_\delta } E^N_x \int_0^{\tau_1} 1_D ( Y^N_s) ds \geq C > 0. 
\end{equation}
We then obtain, following exactly the arguments of \cite{reybellet}, the lower bound 
$$ \nu^N ( D) \geq C \exp ( - N [ \inf_{z \in D } W(z) + h/2 ] ) .$$ 
The proof is completed by an upper bound on $ \nu^N ( \R^n ) , $ which is obtained thanks to Proposition \ref{prop:sigma_0second}.

We finish the above proof by showing  \eqref{eq:last}. Let $Y_0^N =x \in \partial D_\delta .$ Then 
$$ E^N_x \int_0^{\tau_1} 1_D ( Y^N_s) ds = E^N_x \tau_{D^c }  , $$ 
where $ \tau_{D^c} = \inf \{ t \geq 0 : Y^N_t \in D^c \}. $ But for $ Y_0^N = x \in  \partial D_\delta , $ 
\begin{equation}\label{eq:delta}
\delta \le  \| Y_{\tau_{D^c }}^N - x \| \le \sup_{ z \in D} \| b(z) \| \tau_{D^c } + \frac{1}{\sqrt{N}} \sup_{s \le \tau_{D^c } } \| M_s \| , 
\end{equation}
where $ M_s = \int_0^s \sigma( Y_u^N ) d B_u. $ Since the coefficients of $\sigma $ are bounded, using the Burkholder-Davis-Gundy inequality, there exists a positive constant $ \chi $ only depending on the bound of $ b$ on $D$ and on the bounds of $ \sigma $ such that $ E_x^N  \sup_{s \le \tau_{D^c } } \| M_s \| \le \chi  \sqrt{E_x^N \tau_{D^c } }  ,$ and therefore,  
$$ \delta \le \chi \left(  E_x^N ( \tau_{D^c }) + \sqrt{ \frac{E_x^N ( \tau_{D^c })}{N }} \right)    , $$
which in turn implies that 
$$ \inf_N \inf_{ x \in \partial D_\delta } E_x^N ( \tau_{D^c })  = \inf_N \inf_{x \in \partial D_\delta } E^N_x \int_0^{\tau_1} 1_D ( Y^N_s) ds \geq C> 0 $$ 
for a constant $C$ not depending on $N.$

 \hfill $ \bullet $

\section*{Appendix}
{\it Proof of Theorem \ref{theo:stlc}.}
The proof follows the lines of the proof of Theorem 1 in Chapter 6 of Lee and Markus \cite{LeeMarkus}.  As there,  we write 
$ f( x, u) = b(x) +  \sigma ( x) u .$  We fix $x_0 \in \Gamma $ and write 
$$ A (t) =\left(  \frac{\partial f }{\partial x}\right)_{| x= x_t^{x_0}, u = 0} ,  \; B (t)  = \left( \frac{\partial f }{\partial u}\right)_{| x= x_t^{x_0}, u = 0} = \sigma (x_t^{x_0})  .$$ 
Let $ A :=  A ( 0) $ and $B = B(0) .$ 
Then it is easy to see that the columns of $ B, AB, A^2  B, \ldots , A^{n - 1 } B  $ span $ \R^n .$ 
We start by considering the equation 
\begin{equation}\label{eq:homo}
  \dot Y =  A Y +  Bu , Y(0) = 0 .
\end{equation}
Denote by $Y^u (t) , t \le \delta ,  $ a solution to \eqref{eq:homo} driven by $ u(t)  , t \le \delta .$ 
The above system is controllable, since $ B, AB, A^2 B , \ldots , A^{n - 1 } B $ span $ \R^n .$ As a consequence, for every $M $ and for any $ \delta < 1 $ there exist controls $ u_1,  u_2  , \ldots , u_n   $ with $ \| u_i \|_\infty \le M $ such that 
\begin{equation}\label{eq:explicitcontrol}
 Y^{u_1} (\delta)  = r e_1,  \ldots , Y^{u_n } (\delta) = r e_n ,
\end{equation}
where $ e_1, \ldots, e_n $ are the unit vectors of $\R^n $ (Corollary 1 of Chapter 2 of Lee and Markus \cite{LeeMarkus}) and where $r > 0 $ is suitably small.

We wish now to replace the system \eqref{eq:homo} by the time dependent system 
\begin{equation}\label{eq:inhomo}
  \dot W =  A (t)  W +  B (t) u , W(0) = 0 , t \le \delta .
\end{equation}
Write $ W_k ( t) $ for the solution of $ \dot W_k (t) = A ( t) W_k ( t) + B(t) u_k ( t) ,$ where the $u_k (t) $ are given in \eqref{eq:explicitcontrol}. Then $ W_k  (t) $ is explicitly given by 
$$ W_k ( t) = \Phi (t)  \int_0^t \Phi^{-1}  (s) B(s) u_k ( s) ds ,$$
with $ \Phi (t) $ the matrix solution of $ \dot \Phi (t) = A(t) \Phi (t) , $ $ \Phi (0)   = Id.$ Writing $Y_k (t) = Y^{u_k } (t)  , $ we obtain similarly 
$$  Y_k ( t) = \overline \Phi (t) \int_0^t \overline \Phi^{-1} (s) B  u_k ( s) ds ,$$
with $ \overline \Phi (t)  = e^{ A t } $ (recall that $ A = A(0) $). We wish to show that $\|  Y_k ( t) - W_k (t) \| $ is small for $t$ sufficiently small. For that sake, note that there exists a constant $C$ such that for all $ t \le \delta , $ 
$$ \| \Phi (t) \| , \| \overline \Phi (t) \| ,  \| \Phi^{-1} (t)   \| , \| \overline \Phi^{-1} (t)  \| , \| B( t) \| , \| B\| \le C.$$ 
Since
$$ \Phi (t) = Id + \int_0^t A(s) \Phi (s) ds , \; \overline \Phi (t) = Id + \int_0^t A \overline \Phi (s) ds,$$
it follows from this that $ \| \Phi (t) - \overline \Phi (t) \| \to 0 $ as $ t \to 0.$ 

Fix $\varepsilon  > 0 $ such that $ \tilde e_1 , \ldots , \tilde e_n $ still span $ \R^n $ for all $ \tilde e_k \in B_\varepsilon ( r e_k ) , 1 \le k \le n .$ Then there exists $\delta^* $ such that for all $\delta \le \delta^* , $ $ W_k ( \delta ) \in B_\varepsilon ( Y_k ( \delta ) ) ,$ for all $
 1 \le k \le n, $ and therefore the following holds.
\begin{multline}\label{eq:super}
\mbox{The solutions of } \dot W_k (t) = A ( t) W_k ( t) + B(t) u_k ( t) ,\; W_k ( 0 ) = 0, \;  1 \le k \le n , \\
\mbox{ are such that }  W_1 (\delta ) , \ldots , W_n ( \delta ) \mbox{ span } \R^n .
\end{multline}

We are now able to conclude the proof, following the lines of Lee and Markus \cite{LeeMarkus}. Consider $x ( t, \xi )  $ which is the solution of  
$$
 d x( t , \xi ) =  b ( x( t , \xi ) ) dt + \sigma ( x( t , \xi ) ) \dot { h  } (t, \xi)  dt  , \; x(0, \xi ) = x_0, 
$$
following the control $ \dot {h} ( t, \xi ) = \xi_1 u_1 (t) + \ldots + \xi_n u_n ( t) , $ for $ | \xi_i | \le 1, 1 \le i \le n .$ It is clear that $ x ( t, 0 ) = x^{x_0}(t) .$ Hence, if we can prove that $ Z (t) = \left( \frac{\partial x (t , \xi)  }{\partial \xi }\right)_{ | \xi = 0 }$ is non-degenerate at $t = \delta, $ we are done, using the inverse function theorem. But
 
$$ \frac{ \partial x (t, x) }{\partial t} = f( x (t, \xi ) , \dot{h} (t, \xi ) )$$ and thus
$$ \frac{\partial }{\partial t} \frac{\partial x (t, \xi)  }{\partial \xi } = 
  f_x ( x(t, \xi ), \dot h (t, \xi ) ) \frac{\partial x}{\partial \xi } + f_u ( x (t, \xi ) , \dot h  (t, \xi) ) \frac{\partial \dot h }{\partial \xi } . $$ 
Notice that $ x (t, 0) = x^{x_0 }_{t } $ and $\dot h (t, 0) = 0.$ Thus we obtain 
$$ \dot Z  (t) =  A ( t) Z (t) + B(t) U (t) , $$
where $U (t) = (u_1 (t)  , \ldots , u_n (t)   ) .$ Writing $ z_1, \ldots , z_n $ for the columns of $ Z(t), $ this gives
$$ \dot z_k ( t) = A(t) z_k ( t) + B(t) u_k (t) , \; z_k ( 0) = 0 .$$ 
The solutions of this system are given by \eqref{eq:super}, and they are such that $ z_k ( \delta ), 1 \le k \le n ,$ span $\R^n .$ Therefore, $ Z ( \delta) $ is non-degenerate, and this concludes the proof. 
\hfill $\bullet $

{\it Proof of Proposition \ref{prop:sigma_0first}.} The proof follows closely the ideas of Chapter 5.7 of \cite{DZ}.

1) For all $x \in \partial B_\varepsilon (K) , $ by small time local controllability, there exists a smooth path $ \psi^x $ of length $ t^x $ such that $ \psi^x (t^x) \in K= \{x^* \} \cup \bigcup_{l=2}^L K_l $ and such that 
$ \psi^x (t) $ does not leave $ B_{2 \bar \varepsilon /3 }( K) $ for all $ t \le t^x .$ Moreover, this path can be chosen such that $ I_{x, t^x} ( \psi^x ) \le h/2.$ 

2) For all $ x_0 \in K $ there exists $z \in \partial B_\varepsilon (K) $ and a path $ \psi^{x_0} $ of length $ t^{x_0} $ steering $x_0$ to $z, $ during $[0, t^{x_0} ] , $ without leaving $ B_{2 \bar \varepsilon /3}(K) ,$ 
at a cost $ I_{x_0,  t^{x_0}} ( \psi^{x_0 } ) \le h/2 .$  

3) We concatenate the two paths $ \psi^x$ and then $ \psi^{x_0} $ to obtain a new trajectory $ \Psi^x $ of length $T^x = t^x + t^{x_0} $ steering $x$ to $z \in \partial B_\varepsilon (K) .$ Let then 
$$  T_0 := \inf_{ x \in \partial B_{\varepsilon}(K ) } T^x > 0 $$
and put 
$$ {\mathcal O} := \bigcup_{ x \in \partial B_\varepsilon (K) } \{ \varphi \in C ( [ 0, T_0], \R^n ) : \| \varphi - \Psi^x \|_\infty < \varepsilon/ 2  \} ,$$
which is an open set. Then 
$$ \liminf_{N \to \infty } \frac1N \log \inf_{x \in \partial B_\varepsilon (K) } Q^N_x (  {\mathcal O} ) \geq - h ,$$
which implies the assertion since $  Q^N_x ( Y^N \in {\mathcal O} ) \le Q^N_x ( \sigma_0 \geq T_0 ) \le \frac{E^N_x \sigma_0}{T_0} .$ 
\hfill $\bullet$

{\it Proof of Proposition \ref{prop:sigma_0second}.}

1) Let $ S =  \inf \{ t \geq 0 : Y^N \in B_{\varepsilon } ( K) \cup D \} , $ where $ D = (B_{\bar \varepsilon }(K) )^c .$ 
We know by Lemma \ref{lem:1} that there exists $ T_1 > 0 $ such that 
\begin{equation}\label{eq:(1)}
\limsup_{N \to \infty }  \sup_{x \in  \overline {B_{\bar \varepsilon }(K )}} Q_x^N ( S > T_1 ) < 1.
\end{equation}

2) We shall now show that there exists $ T_2$ such that 
\begin{equation}\label{eq:(2)}
\liminf_{N \to \infty } \frac1N \log \inf_{x \in \overline {B_\varepsilon (K)}} Q_x^N ( \sigma_0 \le T_2  ) \geq - h .
\end{equation}
Indeed, like in \cite{DZ}, page 231, we first construct, for all $x \in \overline {B_\varepsilon (K)}$ a smooth path $ \psi^x $ of length $t^x $ such that $ \psi^x (t^x) \in K$ and such that 
$ \psi^x (t) $ does not leave $ B_{2 \bar \varepsilon /3}(K )$ for all $ t \le t^x .$ Moreover, this path can be chosen such that $ I_{x, t^x} ( \psi^x ) \le h/2.$ 

We then fix $ \varepsilon ' > \bar \varepsilon $ such that $ 6 \bar \varepsilon < \varepsilon ' $ and apply Proposition \ref{prop:STLCgamma} to $ 2 \bar \varepsilon $ and $ \varepsilon '.$ This is possible if $ \bar \varepsilon $ is sufficiently small. Then for any $ x_0 \in K$ there exists $ z \in \partial B_{2 \bar \varepsilon }(K) $ and a path $ \psi^{x_0} $ of length $t^{x_0} $ steering $x_0$ to $z , $ during $ [0, t^{x_0} ],$ such that $ I_{x_0, t^{x_0 } } ( \psi^{x_0 } ) \le h/2 .$ We then concatenate the two paths and obtain a new path $ \Psi^x $  of length $ T^x = t^x + t^{x_0} , $ steering $x $ to $z,$ at cost $\le h.$ Let 
$$ T_2 = \sup_{ x \in   \overline {B_\varepsilon (K)} }  T^x < \infty $$
and 
$$ {\mathcal O} =\bigcup_{ x \in \overline {B_\varepsilon (K)} } \{ \varphi \in C ( [ 0, T_2], \R^n ) : \| \varphi - \Psi^x \|_\infty < \bar \varepsilon /  2  \} .$$
Then 
$$ \liminf_{N \to \infty } \frac1N \log \inf_{x \in \overline {B_\varepsilon (K)}} Q_x^N ( Y^N \in {\mathcal O} ) \geq - h ,$$
which implies \eqref{eq:(2)}, since $ \varphi \in {\mathcal O} $ implies that $ \sigma_0 ( \varphi ) \le T_2 .$ 

3) We deduce from the above discussion the following. 
$$ \inf_{ x \in  \overline {B_{\bar \varepsilon}(K ) }} Q^N_x ( \sigma_0 \le T := T_1 + T_2 ) \geq  \inf_{ x \in  \overline {B_{\bar \varepsilon }(K )} } Q^N_x ( S \le T_1 ) \cdot \inf_{ x \in \overline{ B_\varepsilon (K )} } Q^N_x ( \sigma_0 \le T_2 ) =: q .$$ 
By iteration, we obtain 
$$ \sup_{ x \in  \overline {B_{\bar \varepsilon} (K ) }} Q^N_x ( \sigma_0 > k T ) \le (1- q )^k , \mbox{ whence } \sup_{ x \in  \partial {B_{ \varepsilon }(K ) }}  E^N_x \sigma_0 \le \sup_{ x \in  \overline {B_{\bar \varepsilon }(K ) }}  E^N_x \sigma_0  \le \frac{T}{q}.$$  
But 
$$ q \geq e^{ - Nh }   \inf_{ x \in  \overline {B_{\bar \varepsilon }(K )} } Q^N_x ( S \le T_1 ) \geq c e^{- N h } , $$
for $ N $ sufficiently large. This implies the desired assertion. 
\hfill $ \bullet$ 

\section*{Acknowledgments}
I would like to thank an anonymous reviewer for his valuable comments and suggestions which helped me to improve the paper.
This research has been conducted as part of the project Labex MME-DII (ANR11-LBX-0023-01)  and as part of the activities of FAPESP Research,
Dissemination and Innovation Center for Neuromathematics (grant
2013/07699-0, S.\ Paulo Research Foundation).

 \end{document}